\documentclass[11pt]{amsart}
\usepackage{etex}
\usepackage{bbm}
\usepackage{xypic}
\usepackage{tikz}
\usepackage{tikz-cd}
\usepackage{todonotes}
\usetikzlibrary{matrix,arrows,decorations.pathmorphing}
\usepackage{MnSymbol}
\usepackage{latexsym}
\usepackage{pgf}
\usepackage{mathtools}
\usepackage{stmaryrd}
\usepackage{atbegshi}

\usepackage[margin=1in]{geometry}
\usepackage{pictex}
\usepackage{amsmath,amscd}
\usepackage{halloweenmath,enumerate}
\usepackage{afterpage}
\usepackage{youngtab}
\usepackage{psfrag}
\usepackage[boxsize=.25 em]{ytableau}
\usepackage{verbatim}
\usepackage{graphics,graphicx}
\usepackage{changepage}
\usepackage{hyperref}
\usepackage[capitalize, nameinlink, noabbrev, nosort]{cleveref}
\crefformat{section}{\S#2#1#3}

\newtheorem{Thm}{Theorem}[section]
\newtheorem{theorem}[Thm]{Theorem}
\newtheorem{lemma}[Thm]{Lemma}
\newtheorem{corollary}[Thm]{Corollary}
\newtheorem{proposition}[Thm]{Proposition}
\newtheorem{conjecture}[Thm]{Conjecture}
\newtheorem{mainthm}{Theorem}

\theoremstyle{definition}
\newtheorem{example}[Thm]{Example}
\newtheorem{definition}[Thm]{Definition}
\newtheorem{remark}[Thm]{Remark}

\usepackage{pict2e}

\definecolor{babyblue}{rgb}{0.54, 0.81, 0.94}
\definecolor{darkgreen}{rgb}{0.0, 0.3, 0.1 }

\newcommand{\MSB}[1]{\textcolor{violet}{[MSB: #1]}}
\newcommand{\AK}[1]{\textcolor{darkgreen}{[AK: #1]}}

\DeclareMathOperator{\wt}{wt}
\DeclareMathOperator{\supR}{supR}
\DeclareMathOperator{\Grid}{Grid}

\DeclareMathOperator{\Proj}{Proj}
\DeclareMathOperator{\Spec}{Spec}

\DeclareMathOperator{\Perm}{Perm}
\DeclareMathOperator{\Hom}{Hom}
\DeclareMathOperator{\Fl}{Fl}
\newcommand{\RR}{\mathbb{R}}
\newcommand{\CC}{\mathbb{C}}
\newcommand{\ZZ}{\mathbb{Z}}
\newcommand{\PP}{\mathbb{P}}
\newcommand{\NN}{\mathbb{N}}

\newcommand{\mcc}{\mathcal{C}}

\newcommand{\bip}[1]{\mathcal{P}_{#1}}
\newcommand{\momentbp}[1]{\tilde{\mathcal{P}}_{#1}}

\newcommand\bOm{{\bf \Omega}}

\newcommand\defn[1]{{\bf #1}}

\def\wo{w_\circ}
\def\lin{\lambda}
\def\hgt{h}

\def\pbolic^#1{^{<#1>}}
\def\bw{\mathbf{w}}
\def\bc{\mathbf{c}}

\def\bp{\mathbf{p}}

\def\bv{\mathbf{v}}
\def\vl#1{v_{(#1)}}
\def\supp_#1^#2{\supR(#1,#2)}
\def\hc{\hgt^c}
\def\rib_#1{\mathrm{rib}_{#1}}
\def\grid_#1{\Grid(#1)}
\def\path{L}
\def\skew{\nu}
\def\orb_#1^#2{\momentbp{#1}(#2)}
\def\pt{\beta}

\def\hd{\hgt^{\bar{c}}}
\def\AA{{\mathbb A }}
\def\BB{{(\lie{t}_\RR^*)_+}}
\def\calH{{\mathcal H}}
\def\calP{\mathcal{P}}
\def\calX{\mathcal{X}}
\def\mob{\tilde{\mu}}
\def\steep{{\mathrm {inc}}}

\newcommand\Oplus\bigoplus
\newcommand\wh\widehat
\newcommand\Union\bigcup
\newcommand\lie[1]{{\mathfrak #1}}
\newcommand\FF{{\mathbb F}}
\newcommand\QQ{{\mathbb Q}}
\newcommand\tensor\otimes

\newcommand\calY{{\mathcal Y}}
\newcommand\calO{{\mathcal O}}
\newcommand\calF{{\mathcal F}}
\newcommand\calE{{\mathcal E}}
\newcommand\calL{{\mathcal L}}
\newcommand\Gm{{\mathbb G}_m}
\renewcommand\AA{{\mathbb A}}
\newcommand\into\hookrightarrow
\newcommand\otni\hookleftarrow
\newcommand\onto\twoheadrightarrow
\newcommand\iso\cong
\newcommand\junk[1]{}
\newcommand\actson\circlearrowright

\renewcommand{\colon}{\mathrel{:}}
\renewcommand{\colon}{:\mkern\thinmuskip}

\DeclareMathOperator{\conv}{conv}
\title
[Permutahedra, Lusztig varieties, degenerations, and subdivisions]
{Permutahedra, Lusztig varieties, degenerations, \\ and subdivisions}

\author{Allen Knutson}
\author{Mario Sanchez}
\author{Melissa Sherman-Bennett}
\date{\today}

\begin{document}

\begin{abstract}
  We present an embedded (in $G/B$) degeneration of Lusztig varieties
  (which generalize type $A$ Hessenberg varieties) to certain unions of
  Richardson varieties, giving a simultaneous reproof (and extension)
  of results of Anderson--Tymoczko, Harada--Horiguchi--Masuda--Park, and Kim.
  Although torus-equivariant, the degeneration is not Gr\"obner. 
  In the case that the Lusztig variety is the permutahedral toric variety,
  this degeneration provides a dissection of the permutahedron into
  Bruhat interval polytopes, and we prove a more general result showing
  equivariant degenerations of projective toric varieties produce subdivisions
  (as was shown in the Gr\"obner case by Sturmfels) not just dissections.
  A Gr\"obner degeneration would result in a {\em regular} subdivision,
  and despite our degeneration not being Gr\"obner we show in types $A,B,C$
  that our subdivisions of the permutahedron are indeed regular.
\end{abstract}
\maketitle
\setcounter{tocdepth}{2}
	\tableofcontents


\section{Introduction}

\subsection{Overview}

The $(n-1)$-dimensional \defn{permutahedron} $\Pi$ is the convex hull of the
$n!$ permutations, thought of as vectors in $\RR^n$.
The \defn{permutahedral variety} $\Perm$ is the toric variety associated with
$\Pi$; one construction of it is as a generic torus orbit
closure in the type A flag variety $\Fl_n$. 
It is a central object in the
intersection of algebraic geometry and combinatorics, appearing for
instance in the Hodge theory of matroids \cite{huh2014rota}, the study
of moduli spaces \cite{losev2000new},
Chow quotients of the Grassmannian \cite[4.3.10]{kapranov1992chow},
and Hessenberg varieties \cite{lin2024geometry}. 
This paper is inspired by a formula of Anderson--Tymoczko \cite{AT10}
for the cohomology class\footnote{The paper \cite{AT10} contains a more
  general formula for classes of regular semisimple Hessenberg
  varieties.} $[\Perm] \in H^*(\Fl_n)$, and a parallel result of
Harada--Horiguchi--Masuda--Park \cite{HHMP19} on a decomposition of 
$\Pi$. These results involve \emph{Richardson
  varieties} $X_u^v$, which are intersections of Schubert and opposite
Schubert varieties, and their {\em moment polytopes} $\Phi(X_u^v) =: \Pi_u^v$
which are \emph{Bruhat interval polytopes} \cite{KW15}. We summarize
these results in the following theorem.
 
 \begin{theorem}[{\cite[Eq. (14)]{AT10},\cite[Theorem 5.4]{HHMP19}}]\label{thm:prior-work}
 	Let $c=s_{n-1}s_{n-2} \dots s_1 \in S_n$. Then 
 	\[[\Perm] = \sum \{[X_u^{uc}]: \ell(uc) = \ell(u) + \ell(c)\} \]
 	and 
 	\[\Pi = \bigcup \{\Pi_u^{uc} :\ell(uc) = \ell(u) + \ell(c) \}\]
 	is a dissection of the permutahedron, meaning that the polytopes $\Pi_u^{uc}$ have pairwise disjoint interiors.
 \end{theorem}

These results look like they could be a shadow of a single geometric result: an (embedded) torus-equivariant degeneration of $\Perm \subset \Fl_n$ to the union of Richardson varieties
\[\bigcup \{X_u^{uc}: \ell(uc) = \ell(u) + \ell(c)\}.\]
Indeed, such degenerations preserve cohomology classes, and when
torus-equivariant, preserve {\em Duistermaat-Heckman measures} and thereby
moment polytopes. 
So one might expect\footnote{and indeed, we prove in \cref{sec:subdiv-from-degen}} that if $Y$ degenerates to a union $\bigcup X_i$ of toric varieties, then the moment polytopes of the $X_i$ form a  subdivision of the moment polytope of $Y$. 
Lian \cite{Lian24} provided such a degeneration, but did not explain why the moment polytopes $\Pi_u^{uc}$ give a dissection
of the permutahedron.

In this paper, we show that the aforementioned formula for $[\Perm]$,
decomposition of the permutahedron, and degeneration of the
permutahedral variety are in fact part of a larger collection of
results that extend beyond type $A$ and to a more general class of
varieties.

\subsection{Degenerations of Lusztig varieties}

  Let $G$ be a linear algebraic group with Weyl group $W$ and
$\mathbbm{k}$ be an algebraically closed field. Given $x \in G$ and
$w \in W$, the \defn{Lusztig variety} $\mathcal{Y}_w(x)$ is the
variety
 \[\calY_w(x) := \{gB/B \in G/B \; \lvert \; g^{-1} x g \in \overline{BwB}\}.\]
We will focus on the regular semisimple Lusztig varieties,
meaning those where $x$ is a regular semisimple element.

In type $A_n$, all regular semisimple Hessenberg varieties are Lusztig
varieties, and in other types it was recently shown that there are
degenerations of Lusztig varieties to Hessenberg varieties
\cite{BHL25-geomLusztig}. When $w \in W$ is a Coxeter
element\footnote{that is, $w$ has a reduced expression in which each
  simple reflection appears exactly once} and $s$ is regular
semisimple, $\calY_w(s)$ is a permutahedral variety, meaning it is
toric with moment polytope the $W$-permutahedron.

Our main theorem on Lusztig varieties is the following. For $u,w \in W$, we call $uw$ \defn{length-additive} if $\ell(uw)= \ell(u) + \ell(w)$.

\begin{mainthm}\label{mainthm: Degeneration of Lusztig Varieties}
	\begin{enumerate}
		\item For $t \in T$ general enough, we have an embedded degeneration of
$\calY_w(t)$ to the reduced union 
\[\calX_w:=\bigcup\, \{  X_{u}^{uw^{-1}} \ \colon\
u \in W, \  uw^{-1} ~\text{length-additive} \}     \]
in $G/B$.

\item  For $s \in G$ regular semisimple, we have the equality $[\calY_w(s)] = [\calX_w]$ in $K_0(G/B)$ and
\[[\calY_w(s)] = \sum\{[X_{u}^{uw^{-1}}] \ \colon\
u \in W, \  uw^{-1} ~\text{length-additive} \}\]
in $H^*(G/B)$. If $s\in T$, the equalities also hold in $K^T_0(G/B),H^*_T(G/B)$
respectively. 
\end{enumerate}
\end{mainthm}

We draw attention to recent work \cite{bergeron2026coxeterflagvariety}
introducing a space related to $\calX_w$, the union in (1), when $w$ is a Coxeter element $c$:
\[\bigcup\, \{  u^{-1} \cdot X_{u}^{uc^{-1}} \ \colon\
u \in W, \  uc^{-1} ~\text{length-additive} \}.     \]
This union is called the {\em $c$-Coxeter
	flag variety} in \cite{bergeron2026coxeterflagvariety}. The authors of that work use it to give geometric interpretations of
quasisymmetric functions.

We obtain from \cref{mainthm: Degeneration of Lusztig Varieties} many expressions for the class $[\Perm]$, by setting $w^{-1}$ to be a Coxeter element. We also obtain somewhat sharper results for the $K$-class. Below, we use $u* w$ to denote the Demazure product of $u$ and $w$. 

\begin{mainthm}\label{mainthm:permutahedral-class}
   	    Let $c$ be a Coxeter element in $W$. Then in $H^*(G/B)$, 
   \[ [\Perm] = \sum \{ [X_{u}^{uc}] \colon u\in W, uc\text{ length-additive} \}. \]
   In $K_0(G/B)$, 
   \[[\Perm] = \sum (-1)^{\ell(c)-\ell(z) + \ell(y)+1} [X_y^z]\]
   where the sum is over all $y \leq z$ such that $z \leq y *c$ and $[y,z]$ is not contained in a single coset $W_Jw$ of a maximal parabolic $W_J$. 
\end{mainthm}

\subsection{Subdivisions of Coxeter permutahedra}

When $c$ is a Coxeter element, one can show that the union $\calX_c$
from \cref{mainthm: Degeneration of Lusztig Varieties} consists of
toric Richardson varieties. That is, the degeneration
$\calY_c(s) \to \calX_c$ is a semitoric degeneration of a toric
variety. It is well-known that any \emph{Gr\"obner} degeneration of a
projective toric
variety\footnote{such degenerations are always semitoric} gives rise to a
regular subdivision of the moment polytope of the general fiber into
the moment polytopes of the components of the special fiber
\cite{sturmfels1991grobner,zhu2012degenerations}. 
The degeneration of \cref{mainthm: Degeneration of Lusztig Varieties} is
not a Gr\"obner degeneration (specifically, there is no larger torus
acting on the ambient space $G/B$ than already acts on the general
fiber). However, the dissection from \cref{thm:prior-work} indicates
that it still may give rise to a polyhedral subdivision. We show that
this is the case: any semitoric degeneration $Y \to \cup X_i$ of a
toric variety $Y$ gives rise to a subdivision of the moment polytope
of $Y$ into the moment polytopes of the $X_i$ (see
\cref{thm:subdivision-general}). We then use this to obtain the
following combinatorial result.

\begin{mainthm}\label{mainthm: Permutahedron Subdivisions}
    Let $c \in W$ be a Coxeter element. Then the Bruhat interval polytopes
    \[\{\Pi_u^{uc}: u \in W, uc\text{ length-additive}\}\]
    give a subdivision of the $W$-permutahedron $\Pi$. The faces appearing in this subdivision are the Bruhat interval polytopes $\Pi_y^z$ where $z \leq y * c$.
\end{mainthm}
When $c$ is the Coxeter element $s_{n-1} \cdots s_2s_1$ in type $A_n$,
this recovers the HHMP decomposition.

We also show that in type $A$ and types $B/C$, the subdivision in \cref{mainthm: Permutahedron Subdivisions} is \emph{regular}, by exhibiting an explicit height vector which induces the subdivision. The definition of this height vector is inspired by total positivity. Indeed, in type $A_n$, the subdivisions of \cref{mainthm: Permutahedron Subdivisions} correspond to $2^{n-1}$ maximal cones in the positive tropical flag variety, which coincides with the positive flag Dressian (see \cref{ssec:rel-to-matroid-polytopes} for more detail).

The paper is structured as follows. In \cref{sec:degen-and-class-formulas}, we discuss Lusztig varieties and each one's degeneration to a union of Richardson varieties. We conclude formulas for cohomology and $K$-classes, and conjecture an even simpler $K$-class formula. In \cref{sec:subdiv-from-degen}, we show that degenerations of toric varieties give rise to subdivisions of moment polytopes. In \cref{sec:subdiv-permutahedron}, we use the results of \cref{sec:degen-and-class-formulas} and \cref{sec:subdiv-from-degen} to obtain many subdivisions of the permutahedron into Bruhat interval polytopes, one for each Coxeter element $c$. In \cref{sec:regularity-type-A} and \cref{sec:regularity-BC}, we show in types $ABC$ that these subdivisions are regular. \cref{sec:regularity-type-A} also discusses the relation between the subdivisions obtained here and the positive tropical flag variety/positive flag Dressian, and the number of cells in the subdivisions, which depends on the choice of $c$.

  \subsection{Acknowledgements} MSB would like to thank Alejandro Morales for helpful discussions.
  AK was partially supported by the National Science Foundation under
  Award No.~2246959.
  MSB was partially supported by the National Science Foundation under
  Award No.~2444020.

\section{Degenerations and class formulas} \label{sec:degen-and-class-formulas}

In this section, we obtain embedded degenerations of \emph{Lusztig varieties}
in $G/B$, and use this to give formulas for the cohomology and
$K$-classes of Lusztig varieties. As a corollary, we obtain many
semitoric degenerations of the permutahedral variety in $G/B$, and thereby
obtain many formulas for the cohomology and $K$-class of the
permutahedral variety, one for every Coxeter element $c \in W$.

For our purposes, a \textbf{degeneration} of $Z_1$ to $Z_0$ will be a flat
(or more precisely, free\footnote{%
  It is very strange that the adjective ``flat'' is standardly chosen,
  as it allows for disturbing families such as $\Gm \into \AA^1$.
  Really, one wants $\calO_\calF$ to be a free module over the
  pullback of $\calO_{\AA^1}$. This stronger condition is automatic for
  flat {\em projective} families over $\AA^1$, tracing eventually to the fact that
  finitely generated torsion-free modules over a PID are free.})
family $\calF \to \mathbb{A}^1$ where the fiber
$\calF_1$ is isomorphic to $Z_1$ and the ``special fiber''
$\calF_0$ is isomorphic to $Z_0$. (Sometimes one wants all
fibers $\calF_{t\neq 0}$ to be isomorphic, typically in order to state
an inequality relating some level of singularity of $Z_1$ vs. $Z_0$,
but our results are not of that sort and we won't ask this.)
If a group $G$ acts on $\calF$, and the map $\calF\to\AA^1$ is $G$-invariant,
call the family \defn{$G$-invariant}.
If $Z_1$ and $Z_0$ are closed subschemes of $X$,
and the degeneration $\calF$ is a closed subscheme inside the trivial family
$X \times \AA^1 \to \AA^1$, we call $\calF$ an \defn{embedded degeneration}. It will be clear from context what the larger scheme $X$ is.
Most commonly, one obtains embedded degenerations $\calF$ using\footnote{%
  Note that the map $\calF \to \AA^1$ is $\Gm$-{\em equivariant} not
  $\Gm$-{\em invariant}. 
  These degenerations really deserve to be called ``Gr\"obner degenerations'',
  though that traditionally covers only the case where $X = \AA^n$
  and $\Gm$ acts by diagonal matrices, corresponding to a term order
  determined by an integer weighting on the variables.}
a $\Gm$-action on $X$ (hence diagonally on $X\times \AA^1$),
taking $\calF := \overline{\Gm\cdot (Z_1 \times \{1\})}$;
however our later embedded degenerations will {\em not} be of this sort.

Our interest in embedded degenerations stems from the following result.

\junk{
\begin{proposition}\label{prop:sameclass}
  Let $X$ be a projective scheme with two closed subschemes
  $Z_1$ and $Z_0$. If there is an embedded degeneration from $Z_1$ to $Z_0$,
  then
  \[[Z_1] = [Z_0] \quad \text{in $A_*(X)$ or $H_*(X)$}, \]
  and
  \[[\mathcal{O}_{Z_1}] = [\mathcal{O}_{Z_0}] \quad \text{in $K_0(X)$}.
    \qquad\qquad\quad
  \]
  If $\calF$ is a $G$-invariant degeneration, then the equalities hold also for
  the equivariant classes in $K^G_0(X),A_*^G(X),H_*^G(X)$.
\end{proposition}
}

\begin{proposition}\label{prop:flatness}
  Let $\calF \subseteq X \times \AA^1$ be closed, and $X$ be projective.
  Then $\calF$ is flat over $\AA^1$ if and only if each fiber defines
  the same class in $K_0(X)$. In this case the fibers also define the same
  class in $A_*(X)$ and (in the case $\FF=\CC$) in $H_*(X)$.  
  If $\calF$ is a $G$-invariant degeneration, then all fibers define the same equivariant classes in $K^G_0(X),A_*^G(X),H_*^G(X)$.

  If the fibers are equidimensional, reduced, and have the same $A_*(X)$-class, then the family is flat. The same holds for families $\calF$ with a {\em finite map} to
  $X\times\AA^1$, not necessarily an inclusion
  (a.k.a. a family of ``branchvarieties'' \cite{alexeev2010complete}
  rather than subschemes). 
  
\end{proposition}

We will not make direct use of the latter paragraph (which is closely related to \cite[Lemma 1.7.5]{knutson2005grobner}),
but it was easy enough to include. We note that in the latter paragraph, all fibers having the same $A_*(X)$-class is \emph{not} enough to imply the family is flat.

\begin{proof}
  Assume first that the family is flat.
  Flatness at $t \in \AA^1$ implies that the pullback of
  $[\calF] \in K_0(X \times \AA^1)$ along the inclusion
  $X\times \{t\} \into X\times \AA^1$ gives $[Z_t] \in K_0(X\times \{t\})
  \iso K_0(X)$.
  Since $[\calF]$ pulls back both to $[Z_1]$ and $[Z_0]$, they are equal.

  The $A_*$ equality follows from the $K_0$ equality by taking associated
  graded, and then follows in homology by applying the cycle map
  $A_*(X) \to H_*(X)$. For details of the equivariant cycle map we refer
  to \cite{edidin1998equivariant}.

  Now assume the family is not flat; we want to show that the classes
  in $K_0(X)$ are not the same for all fibers. Since $X$ is projective,
  we can push forward those classes along $X\into \PP^N$, and it
  will suffice to show the classes differ in $K_0(\PP^N)$.
  Note that the map $[\calE] \mapsto (d \mapsto \chi(\PP^n;\, \calE(d)))$
  taking a coherent $\calO_{\PP^n}$-module $\calE$ to its
  Hilbert polynomial defines a linear isomorphism
  $K_0(\PP^N) \to \{\ZZ$-valued polynomials of degree $\leq N\}$.

  Define the subfamily $\calF' \subseteq \calF$ as the closure of the
  ``generic fiber'' (the preimage of the generic point of $\AA^1$),
  automatically flat, and let $s\in\AA^1$ be
  a general enough point that $\calF'_s = \calF_s$. 
  Since $\calF$ was assumed not flat,
  it properly contains $\calF'$. Hence it has some primary component
  whose image in $\AA^1$ misses the generic point, i.e. maps to a some
  closed point $t\in \AA^1$.  Then $\calF'_t$ properly contains
  $\calF_t$, so their Hilbert polynomials differ. But now
  $\calF'_t \neq \calF_t = \calF_s = \calF'_s$, i.e. the Hilbert
  polynomials differ, as was to be shown.

  If $\calF$ is a constant family of plane curves, but with an embedded
  point added to one fiber, then the fibers have the same $A_*(\PP^2)$-class
  despite the family not being flat.

  Finally, assume that the fibers of $\calF$ are reduced and equidimensional
  and define the same class in $A_*(X)$. Assume for contradiction that
  $\calF$ is not flat, and construct $\calF'_t \subseteq \calF_t$
  as above. Again let $\calF_s$ be a general enough fiber that
  $\calF'_s = \calF_s$. Then we have equalities of Chow classes
  $[\calF_t] \stackrel{(1)}= [\calF_s] \stackrel{(2)}= [\calF'_s]
  \stackrel{(3)}= [\calF'_t]$, based on (1) the assumption (2) the
  generality of $s$ (3) the flatness of $\calF'$. Consequently any
  primary component $C$ of $\calF'_t$ not in $\calF_t$ must be of lower
  dimension. By the equidimensionality assumption, $C$ must be embedded.
  By the reducedness assumption, $C$ cannot exist, contradiction.

  To extend the result to families of branchvarieties, we point out
  that a branchvariety of $\PP^n$ has an associated $K$-class (or Hilbert
  polynomial), and a proper inclusion of branchvarieties gives an
  inequality on their Hilbert polynomials. Then give the same proof as before.
\end{proof}

We include another standard result about families, that is not relevant
for our class formulas but will show up when we discuss subdivisions.

\begin{theorem}[Zariski's Main Theorem, usual formulation]
  Let $F' \to C$ be a birational map from a variety $F'$ to a normal
  variety $C$, with finite fibers. Then it is an isomorphism to an open
  subset of $C$.
\end{theorem}

\begin{corollary}\label{cor:ZMT}
  Let $F$ be a variety over an algebraically closed field.
  Let $\tau\colon F \to \AA^1$ (or any other normal variety)
  be a proper map whose general fibers are
  connected (in particular, nonempty). Then {\em every} fiber is connected.
\end{corollary}

\begin{proof}
  We Stein factorize $\tau\colon F\to \AA^1$ as $F \to F' \to \AA^1$, i.e.
  $F'$ is the global spectrum ${\bf Spec}\, \tau_* \calO_F$ of the
  pushforward of the structure sheaf. This replaces each fiber by
  its affinization, not changing the number of connected components.

  Since $F\to \AA^1$ is proper,
  the fibers of $\tau'\colon F' \to \AA^1$ are now proper and affine,
  hence finite. Since the general fibers of $\tau$ were connected,
  the general fibers of $\tau'$ are too. Being finite and connected,
  the general fibers are points. So over an open set, our map $\tau'$
  is bijective; since our base field is algebraically closed,
  our map $\tau'$ is birational.

  We can finally apply Zariski's Main Theorem to say that $\tau'$ is
  an isomorphism to an open set. But now since $\tau$ is proper, so is
  $\tau'$, so that open set is all of $\AA^1$. Now that every fiber of
  $\tau'$ is a point, every fiber of $\tau$ was connected.
\end{proof}

\subsection{Single and double Schubert varieties} 

Let $G$ be a semisimple algebraic group, with Borel subgroup $B$, opposite Borel $B_-$, maximal torus $T =B \cap B_-$ and Weyl group $W=N(T)/T$. The Weyl group has $r$ simple transpositions, indexed by the vertex set $I$ of the corresponding Dynkin diagram. We write $G_{\Delta}$ and $T_{\Delta}$ for the diagonal embedding of $G \hookrightarrow G \times G$ and $T \hookrightarrow T \times T$ respectively.
 Recall that for $ w \in W$, the \defn{Schubert variety} and \defn{opposite Schubert variety} in $G/B$ are, respectively, 
\[X^w := \overline{BwB}/B \quad \text{and} \quad X_w := \overline{B_- w B}/B.\]
They have dimension $\ell(w)$ and $\ell(w_0)-\ell(w)$, respectively. For $u\leq  v \in W$, the \defn{Richardson variety} is 
\[X^v_u:= X_u \cap X^v\]
and has dimension $\ell(v)- \ell(u)$. If we shrink the Schubert varieties in the above definition to Schubert cells, we obtain the \defn{open Richardson variety} $(X^v_u)^\circ$.

Define the \textbf{double Schubert variety} $\bOm_w \subseteq G/B \times G/B$ by
\begin{align*}
    \bOm_w &:= \overline{G_\Delta \cdot (wB/B,B/B)} 
    = \{(gB/B,hB/B)\colon h^{-1}g \in \overline{BwB} \}
\end{align*}

These are $G_{\Delta}$-varieties hence $T_{\Delta}$-varieties, inheriting the
$T_{\Delta}$-action from $(G/B)^2$. For use later, we describe an
embedded degeneration of $\bOm_w$, a strengthening of
the homology statement \cite[Proposition 3.8]{K20-deligne-lusztig}.

\begin{proposition}\label{prop: degeneration of double Schubert}
  \cite[Theorem 5.7, Lemma 5.10]{GabeRaj}
  Let $\sigma\colon \Gm \to T$ be a regular dominant coweight,
  and define the $T_\Delta$-invariant flat family
  \[ \calF := \overline{ \{(E,\sigma(z)\cdot F,z)\colon (E,F)\in \bOm_w,
    z \in \Gm\}} \ \subseteq (G/B)^2 \times \AA^1 \]
  Then the zero fiber is the reduced union
  $\bigcup \{X_u \times X^v \; \lvert \; v = uw^{-1} \text{ length-additive}\}$,
  which is Cohen-Macaulay.
\end{proposition}

\subsection{Lusztig varieties}

Given an element $x \in G$ and $w \in W$, the \textbf{Lusztig variety}
\cite{K20-deligne-lusztig} is the variety
 \[\calY_w(x) := \{gB/B \in G/B \; \lvert \; g^{-1} x g \in \overline{BwB}\}.\]
 Shrinking $\overline{BwB}$ to $BwB$ in the above definition yields the
 \textbf{open Lusztig variety} $\calY_w^\circ(x)$; we note that $\overline{\calY_w^\circ(x)}= \calY_w(x)$, as follows from \cite[Lemma 1.1]{Lus79} (see also \cite[Proposition 2.10]{BHL25-geomLusztig}).
 \junk{\AK{It is in no way a cell. Also, do we use it for anything?} 
 	\MSB{I was borrowing the terminology from \cite{BHL25-geomLusztig}, happy with your version instead. We only need it because Kim works with it and its closure instead of Lusztig varieties as we've defined them, so I wanted to justify that his results apply to us. I've streamlined.}}

 We will study $\calY_w(x)$ by relating it to double Schubert varieties. 

\begin{lemma}\label{prop:transverse-intersections-of-double-Schuberts}
The Lusztig variety $\calY_w(x)$ is the projection of 
\[(x, 1)\bOm_{e} \cap \bOm_w \subseteq G/B \times G/B\] to the second
factor. If our base field is characteristic $0$ then
for general enough $x\in G$, or general enough $x\in T$, the
intersection above is transverse.  \junk{Allen, should ``generic" be
  something else here?  Also it would be a lot nicer to say
  ``including for all regular semisimple $x$" but I do not know how to
  get this from Kleiman transversality directly.}
\end{lemma}

\begin{proof}
  We have 
  \begin{eqnarray*}
    (x, 1) \bOm_e \cap \bOm_w
    &=& \{(gB/B, hB/B): gB = xhB,\ h^{-1}g \in \overline{BwB} \}  \\
    &=& \{(gB/B, hB/B): gB = xhB,\ h^{-1}xh \in \overline{BwB} \}  
  \end{eqnarray*}
  Since $gB$ can be determined from $hB$, under the projection map
  $(G/B)^2\to G/B$ this variety embeds into the second factor, and its
  image is clearly the Lusztig variety.

  By Kleiman transversality (which requires characteristic $0$,
  but perhaps the techniques from \cite{SchubertInduction} could be
  used instead), the
  intersection $(a,b)\bOm_e \cap \bOm_w$ is transverse so long as
  $(a,b)$ is a general enough element of $G\times G = (G\times 1)G_\Delta$.
  Applying $(b^{-1},b^{-1})\in G_\Delta$ to both intersectands, the latter
  of which doesn't move, we find $(b^{-1}a,1)\bOm_e \cap \bOm_w$ is transverse.
  Taking $x = b^{-1}a$, we learn that for general enough $x\in G$
  the intersection is transverse.

  Now assume $x,y \in G$ are conjugate, i.e. $x = g y g^{-1}$.
  Then $(x,1) = (g,g)(y,1)(g,g)^{-1}$, so
  \[ (x,1)\bOm_e \cap \bOm_w
  = (g,g)(y,1)(g,g)^{-1}\bOm_e \cap \bOm_w
  = (g,g)(y,1)\bOm_e \cap \bOm_w
  = (g,g)\left( (y,1)\bOm_e \cap \bOm_w\right)
  \]
  so $(x,1)\bOm_e \cap \bOm_w$ is transverse iff
  $(y,1)\bOm_e \cap \bOm_w$ is transverse. General enough elements of
  $G$ are regular semisimple, hence conjugate to general enough torus elements,
  proving the final claim.
\end{proof}

For the remainder of the section, we will focus on the varieties
$\calY_w(s)$ where $s$ is a regular semisimple element. 
Conjugating $s$ has the effect of translating $\calY_w(s)$, as noted
in the proof above, so we often assume $s \in T$, in which case $\calY_w(s)$ is invariant under the left action of $T$ on $G/B$.

The next result collects some facts on $\calY_w(s)$ where $s$ is regular semisimple. The first is \cite[Theorem 1.1]{BHL25-geomLusztig}, and the second and third are \cite[Theorem 4.6, Lemma 4.8]{K20-deligne-lusztig}.

\begin{proposition}\label{thm: Lusztig facts}
    Let $s \in G$ be a regular semisimple element and let $w \in W$. Let $W'$ be the subgroup of $W$ generated by the set of simple reflections in any reduced expression for $w$. Then

    \begin{enumerate}
    \item $\calY_w(s)$ is normal, Cohen--Macaulay, of pure dimension
      $\ell(w)$ and has at worst rational singularities.
    \item The number of connected components of $\calY_w(s)$ is $|W/W'|$.
    \item The class $[\calY_w(s)] \in H^*(G/B)$ does not depend on the
      regular semisimple element $s$.
    \end{enumerate}
\end{proposition}

\begin{remark}\label{rem:Hessdegen}
	In type $A$, every regular semisimple Hessenberg variety is a regular
	semisimple Lusztig variety; see \cite{BHL25-geomLusztig} for more
	details. Outside of type $A$, \cite[Theorem 5.12]{BHL25-geomLusztig}
	shows that regular semisimple Hessenberg varieties are degenerations
	of regular semisimple Lusztig varieties. 
\end{remark}

Lusztig varieties include certain toric varieties, which we now recall.

\begin{definition}
	Let $gB/B \in G/B$. The $T$-orbit closure $\overline{T \cdot gB/B}$ is
	called a \defn{permutahedral variety} if its moment polytope is the
	$W$-permutahedron; equivalently, if $gB/B \in \bigcap_{w\in W} wB_-B/B$.
	The cohomology class
	$[\overline{T \cdot gB/B}] \in H^*(G/B)$ is called the
	\defn{permutahedral class}, which we also denote by $[\Perm]$.
\end{definition}

The permutahedral class is well-defined; that is, if
$\overline{T \cdot gB/B}$ and $\overline{T \cdot hB/B}$ are permutahedral
varieties, then $[\overline{T \cdot gB}] = [\overline{T \cdot hB}]$
(as is most easily checked in equivariant cohomology).

In type $A$, 
$\calY_{c}(t)$ for $c = s_{n-1}s_{n-2}\cdots s_1$
is well-known to be a permutahedral variety. The next proposition characterizes the Lusztig varieties which are permutahedral varieties. Recall that a \defn{(standard) Coxeter element} of $W$ is an element with a reduced word using each simple reflection exactly once.

\begin{proposition} \label{prop:lusztig-permutahedral}
	Let $w \in W$ and let $t\in T$ be a regular semisimple element. Then
	$\calY_w(t)$ is a permutahedral variety if and only if $w$ is a
	Coxeter element. 
	
	For $s\in G$ regular semisimple and $\calY_w(s)$ irreducible, $[\calY_w(s)]= [\Perm]$ if and only if $w$ is a Coxeter element.
\end{proposition}

\begin{proof}
	It follows from the definition that for any $w$, $\calY_w(t)$ is
	$T$-invariant. Further, $\calY_w(t)$ contains all $T$-fixed points
	of $G/B$. Indeed, the $T$-fixed points of $G/B$ have the form
	$\dot u B$ where $\dot u \in N(T)$. By the definition of normalizer,
	$\dot u^{-1} t \dot u =t $. As $t \in B \subset \overline{B w B}$,
	we have $\dot u \in \calY_w(t)$ for all $u \in W$.
	
	By definition, permutahedral varieties are irreducible and have
	dimension $r$ (the rank of $G$). Using \cref{thm: Lusztig facts} (2) and (3),
	$\calY_w(t)$ has these properties if and only if a reduced word for
	$w$ uses each simple transposition exactly once, i.e. if and only if
	$w$ is a Coxeter element. So if $w$ is not a Coxeter element,
	$\calY_w(t)$ is not a permutahedral variety.
	
	Conversely, if $w$ is a Coxeter element, $\calY_w(t)$ is an irreducible $T$-invariant variety of dimension $r$. Its moment polytope is the convex hull of the images of its $T$-fixed points, which are exactly $\{\dot w B\}_{w \in W}$. So its moment polytope is the $W$-permutahedron. As the $W$-permutahedron is $r$ dimensional, $\calY_w(t)$ must contain an $r$ dimensional torus orbit, which is dense. This shows $\calY_w(t) = \overline{T \cdot X}$ for some $X$ and its moment polytope is the permutahedron, as desired.
	
	For the cohomological statement, the $\impliedby$ direction is implied by the first part of the proposition and \cref{thm: Lusztig facts} (3). The $\implies$ direction follows from the aforementioned fact that $\calY_w(s)$ is irreducible and dimension $r$ if and only if $w$ is a Coxeter element.
\end{proof}

\junk{\MSB{Is it true that $[\calY_w(s)] = [\Perm]$ is the permutahedral
		class if and only if $w=c$? For all other $w$, either
		$\dim \calY_w(s)> n-1$ or $\calY_w(s)$ is reducible. I think the
		former clearly rules out that $[\calY_w(s)] = [\Perm]$, but not as
		clear about the latter.}
	\AK{I'm not seeing a trivial way to rule it out, what with not having a
		particularly computable formula for the class as a combination
		of Schubert classes. Trivially, the two are different in {\em $T$-equivariant}
		cohomology. I think one could maybe show that $[Perm]$ times a
		certain product of $c_1$s = $[\calY_w(t)]$ times a different product
		of $c_1$s. Anyway, the basic issue is that for $w$ not using all
		the generators, $\calY_w(t)$ is disconnected, i.e. not a ``variety'',
		so shouldn't we just say those are dumb and disallow them?}
	\MSB{Yeah, seems good.}}

\subsection{The class of a Lusztig variety}\label{ssec:class-of-lusztig-var}

In this section, we give a degeneration of $\calY_w(t)$ to a union of
Richardson varieties, and use this to deduce formulas for the
cohomology class, $T$-equivariant cohomology class, and $K$-class of $\calY_w(s)$ as a
sum of Richardson classes. We note that the expression for the ordinary
cohomology class $[\calY_w(s)] \in H^*(G/B)$ was previously found by
Kim \cite[Theorem 4.12]{K20-deligne-lusztig}.

For $u, v \in W$, we say that $uv$ is \defn{length-additive} if $\ell(uv)=\ell(u) + \ell(v)$. 
For fixed $w \in W$, we let 
\[
  \calP(w):=\{[y, z]: [y,z] \subset [u, uw^{-1}] \text{ for some $u$ with $uw^{-1}$ length-additive}\},
\]
which we regard as a poset with respect to containment\footnote{One
  can give an alternate definition of this poset using the Demazure
  product, as we do in the next subsection}. For a poset $\calP$, let
$\mob_{\calP}: \calP \to \ZZ$ be the unique function\footnote{We note
  that if $\hat{\calP}$ is the poset obtained from $\calP$ by adding
  minimal and maximal elements $\hat{0}$ and $\hat{1}$ and
  $\mu_{\hat{\calP}}$ is the M\"obius function of $\hat{\calP}$, then
  $\mob_{\calP}(b) = \mu_{\hat{\calP}}(b, \hat{1})$.} satisfying
\[\sum_{b \geq a} \mob_{\calP}(b)=1\]
for all $a \in \calP$.

\begin{theorem}\label{thm:Hclass}
  \begin{enumerate}
  \item   For $t \in T$ general enough, we have an embedded degeneration of
    $\calY_w(t)$ to the reduced union 
    \[\calX_w:=\bigcup\, \{  X_{u}^{uw^{-1}} \ \colon\
      u \in W, \  uw^{-1} ~\text{length-additive} \}     \]
    in $G/B$.
  \item  For $s \in G$ regular semisimple, we have the equality
    \[[\calY_w(s)] = \sum\{[X_{u}^{uw^{-1}}] \ \colon\
      u \in W, \  uw^{-1} ~\text{length-additive} \} \]
    in $H^*(G/B)$. If $s\in T$, the equality holds also for
    the $H^*_T(G/B)$-classes.
  \item Let $\mob:= \mob_{\calP(w)}$, as was just introduced. For $t \in T$ general enough, 
we have the equalities 
\begin{align*}[\calY_w(t)] &= [\bigcup \{X_{u}^{uw^{-1}}\ \colon 	u \in W, \  uw^{-1} ~\text{length-additive}\}]\\
		  &= \sum_{[y,z] \in \calP(w)} \mob([y,z]) [X_y^z]\end{align*}
		in $K_0(G/B)$. 
              \end{enumerate}
\end{theorem}

\begin{proof}
  We first prove (1). We have seen that
  \[ \calY_w(t) = \pi_2((t , 1) \bOm_1 \cap \bOm_w)
  \]
  where $\pi_2$ is the projection onto the second component and the
  intersection above is transverse.

  Take the family from Proposition \ref{prop: degeneration of double Schubert}
  and intersect fiberwise with $(t,1)\bOm_1$. We want to know the resulting
  family $\calF$ is flat at $z=1$ and $z=0$. If $Q \subset \AA^1$ is
  the set of points over which $\calF$ has torsion, define
  $\calF' := \overline{ \{(x,z) \in \calF\colon z \notin Q\setminus\{0,1\} \}}$
  which is automatically flat except possibly at $0,1$.
 
  We know $\calF'$ is flat at $z=1$ because $t$ was chosen to make the
  intersection transverse. 
  We can use Proposition \ref{prop:flatness} to check flatness,
  once we establish that $[\calF'_0] = [\calF'_1] \in K_0((G/B)^2)$. In $K_0((G/B)^2)$, we have
  \[
    [\calF'_1]
    = [\calF_1] 
    = [(t,1)\bOm_1 \cap \bOm_w] 
    = [(t,1)\bOm_1] \, [\bOm_w] 
    = [(t,1)\bOm_1] \, \left[
      \Union \left\{ X_{u}^v \colon v= uw^{-1}\text{ length-additive} \right\} \right]
  \]
  The first equality holds because $\calF,\calF'$ agree at $1$
  (also at $0$ and the generic point).
  The second is from the definition of $\calF$.
  The third is from the transversality, i.e. the choice of $t$. 
  The last is from Proposition \ref{prop: degeneration of double Schubert},
  which also tells us that union is Cohen--Macaulay. 

  To reassemble the final product as an intersection, we use the fact
  from \cite[proof of Proposition 7.1]{fulton2013intersection}
  that when taking the proper intersection 
  of two Cohen--Macaulay subschemes of a smooth scheme, the $K$-classes multiply.
  We continue:
  \[ = \left[ (t,1)\bOm_1 \cap
      \bigcup \{ X_{u}^v \colon v= uw^{-1}\text{ length-additive} \}
    \right]
    = [\calF_0] = [\calF'_0]
  \]
  Now Proposition \ref{prop:flatness} applies to $\calF'$, making it a
  degeneration (embedded in $(G/B)^2$) of $(s,1)\bOm_1 \cap \bOm_w$ to
  \begin{align*}
    (t , 1) \bOm_1 \cap \left ( \bigcup X_u \times X^v \; \lvert \; v = uw^{-1} \text{ length-additive} \right ). 
  \end{align*}
  
  The points in this intersection are $(tg,g)$ where $tg \in X_u$ and
  $g \in X^v$. As $X_u$ is closed under the action of $T$, 
  together these imply that $g \in X_u^v$. Conversely, if
  $g \in X_u^v$ for $v=uw^{-1}$ length-additive, then $(tg,g)$ is in
  the intersection above. Therefore, if we project this flat family,
  remembering that the projection is an isomorphism, we obtain a flat
  degeneration of $\calY_w(t)$ to
  \[\Union \{ X_{u}^v\colon u,v\text{ such that }v= uw^{-1}\text{ is length-additive} \} \]
  embedded in $G/B$.
  This shows (1).
  
  For (2), by \cref{thm: Lusztig facts}, the class of $[\calY_w(s)]$ does not
  depend on the choice of regular semisimple element, so (2) follows from \cref{prop:flatness} and (1) (as we may choose $t$ to be some
  regular semisimple element).
  
  For (3), the first equality follows from (1) and \cref{prop:flatness}.
  The second equality follows from applying \cite[Theorem 1]{K09} to
  the union $\calX$, which is compatibly split w.r.t. the standard Frobenius
  splitting on $G/B$ (that splits exactly the unions of Richardson varieties).
\end{proof}

\begin{corollary}[{\cite[Equation (14)]{AT10}}]
 Let $w\in S_n$ be a dominant permutation,
    so $\calY_w(t)$ is a regular semisimple Hessenberg variety.
    Then
    \[ [\calY_w(t)] =
      \sum \{ [X_{u}^{uw^{-1}}] \colon u\text{ with }uw^{-1}\text{ length-additive} \} \]
\end{corollary}

In the special case when $w$ is a Coxeter element, we obtain results on the cohomology and $K$-class of the permutahedral variety. 

\begin{theorem}\label{thm:k-cohom-formulas-for-perm}
	    Let $c$ be a Coxeter element in $W$. Then in $H^*(G/B)$, 
	\[ [\Perm] = \sum \{ [X_{u}^{uc}] \colon u\text{ with }uc\text{ length-additive} \}. \]
	In $K_0(G/B)$, 
	\[[\Perm] = \sum (-1)^{\ell(c)-\ell(z) + \ell(y)+1} [X_y^z]\]
	where the sum is over all $[y,z] \in \calP(c)$ such that $[y,z]$ is not contained in a single coset of $W_J \backslash W$ for any maximal parabolic subgroup $W_J$.
\end{theorem}
The cohomological result here is immediate from \cref{thm:Hclass}. We delay the proof of the K-theoretic statement to the next section.

\subsection{Conjectural $K$-class formulas for Lusztig varieties}
Here, we give a conjectural sharpening of \cref{thm:Hclass} (3), and in particular a conjectural formula for $\mob([a,b])$. We can prove this conjecture when $w$ is a Coxeter element, which is the content of \cref{thm:k-cohom-formulas-for-perm}.

 For the remainder of the subsection, we fix $w \in W$, the poset $\calP(w)$ and function $\mob: \calP(w) \to \ZZ$ as \cref{ssec:class-of-lusztig-var}. We note that $\calP(w)$ is also the poset of Richardson varieties appearing in the union $\calX_w$ in \cref{thm:Hclass} (1), ordered by containment. It is a graded poset of rank $\ell(w^{-1})= \ell(w)$.
 
 We have an explicit conjecture for $\mob([y,z])$, which requires some terminology. We call the elements $[y,z]$ of $\calP(w)$ the \defn{faces} of $\calP(w)$. A codimension 0 face is a \defn{facet} and a codimension 1 face is a \defn{ridge}.
 
 \begin{lemma}\label{lem:ridge-in-one-or-two-facets}
 	Every ridge of $\calP(w)$ is contained in either one or two facets.
 \end{lemma}
 \begin{proof}
 	If $[y,z]$ is a ridge, then the only possibilities for facets containing it are $[y, yw^{-1}]$ and $[zw, z]$.
 \end{proof}
 
 A ridge is \textbf{exterior} if it is contained in exactly one facet, and is \textbf{interior} if it is contained in exactly two facets; by the above lemma, all ridges are either interior or exterior. A face of codimension at least 2 is \textbf{exterior} if it is contained in some exterior ridge and is otherwise \textbf{interior}.
 
 \begin{conjecture}\label{conj:mobius-fcn-and-k-thry-formula}
 	We have 
 	\[\mob([y,z]) = \begin{cases}
 		0 & \text{if $[y,z] \in \calP(w)$ exterior}\\
 		(-1)^{\ell(w)-\ell(z)+ \ell(y)} & \text{if $[y,z] \in \calP(w)$ interior}
 	\end{cases}\]
 	and so in $K$-theory,
 	\[[\calY_{w}(s)] = \sum_{\substack{[y,z]\in \calP(w) \\ \text{ interior }}} (-1)^{\ell(w)-\ell(z)+ \ell(y) +1} [X_y^z].\]
 \end{conjecture}
 
 We prove \cref{conj:mobius-fcn-and-k-thry-formula} in the case that $w$ is a Coxeter element, using the following proposition. The \defn{face poset} of a finite cell complex is a poset whose elements are the cells $\sigma$, with order relation given by containment of closures.
 
 \begin{proposition}[{\cite[Proposition 3.8.9]{Stanley-EC1}}]\label{prop:mob-fcn-for-face-poset}
 	Suppose $\calP$ is the face poset of a finite regular cell complex whose underlying topological space is a manifold $\Gamma$, possibly with boundary. Then for a cell $\sigma$ of codimension $d$,
 	\[\mob_{\calP}(\sigma) = \begin{cases}
 		0 & \text{if $\sigma$ lies on the boundary of }\Gamma\\
 		(-1)^{d+1} & \text{otherwise.}
 	\end{cases}\]
 \end{proposition}
 
 \begin{proof}[Proof of \cref{thm:k-cohom-formulas-for-perm}]
 	\cref{thm:subdivision-Coxeter} below shows that there is a polyhedral subdivision of the $W$-permutahedron $\Pi_e^{w_0}(\pt)$ whose face poset is isomorphic to $\calP(c)$. The cell (closure) corresponding to $[y,z] \in \calP(c)$ is the Bruhat interval polytope $\Pi_y^z(\pt)$. Polyhedral subdivisions are regular finite cell complexes, and the $W$-permutahedron is a manifold with boundary. We apply \cref{prop:mob-fcn-for-face-poset} to determine $\mob_{\calP(c)}$. 
 	
 	An interval $[y,z] \in \calP(c)$ is exterior exactly when the corresponding cell closure $\Pi_y^z(\pt)$ is contained in the boundary of the $W$-permutahedron, and is interior exactly when $\Pi_y^z(\pt)$ intersects the interior. As $\Pi_e^{w_0}(\pt)$ is a polytope, the former situation occurs if and only if $\Pi_y^{z}(\pt)$ is contained in a single facet. This in turn occurs if and only if the interval $[y,z]$ is contained in a single coset of $W_J\backslash W$ for some maximal parabolic subgroup $W_J$. 
 	
 	The dimension of $\Pi_e^{w_0}(\pt)$ is $\ell(c)$, and the dimension of the cell closure $\Pi_y^z(\pt)$ is $\ell(z) - \ell(y)$. Combining this with the above paragraph yields the theorem statement.
 \end{proof}

 \begin{remark} As is clear from the proof of \cref{thm:k-cohom-formulas-for-perm}, \cref{conj:mobius-fcn-and-k-thry-formula} holds if one can apply \cref{prop:mob-fcn-for-face-poset}.
For arbitrary $w$, the poset $\calP(w)$ is the face poset of a finite regular CW complex. Namely, the totally nonnegative flag variety 
 	\[(G/B)^{\ge 0} = \bigsqcup_{p\leq q}(X_p^q)^{>0} =\bigcup_{p\leq q}(X_p^q)^{\ge 0} \]
 	is a regular cell complex \cite{GKL22} with face poset $\{[y, z]: y \leq z\}$ ordered by inclusion. The sub-complex
 	\[\bigcup_{\substack{u: ~uw^{-1}\\\text{length-additive}}} (X_u^{uw^{-1}})^{ \ge 0}= \bigsqcup_{[y,z] \in \calP} (X_y^z)^{>0}\]
 	is also regular and has face poset $\calP(w)$. To apply \cref{prop:mob-fcn-for-face-poset}, one would need to show this sub-complex is a manifold with boundary. Another strategy to show \cref{conj:mobius-fcn-and-k-thry-formula} would be to show that the poset obtained from $\calP(w)$ by adding a minimal and maximal element is shellable.
 \end{remark}
 
If \cref{conj:mobius-fcn-and-k-thry-formula} is true, it would be useful to have an explicit characterization of the interior and exterior elements of $\calP(w)$. We give this below, though we first need an alternate description of $\calP(w)$.
 
 We use $v*w$ to denote the \defn{Demazure product} of $v$ and $w$, which is the unique Bruhat-maximal element of $\{v'w: v' \leq v\}$. The Demazure product is associative and can also be computed recursively by the formula
 \[v * s_i = \begin{cases}
 		v s_i & \text{if}~ vs_i > v\\
 		v & \text{else}.
 	\end{cases}\]

\begin{lemma}\label{lem:interval-poset-Demazure-prod}
For $w \in W$, $\calP(w)= \{[y,z]: z \leq y * w^{-1}\}.$
\end{lemma}
\begin{proof} Denote the right-hand side by $\calP$.
For the containment $\supset$, consider $[y,z] \in \calP$. We have that $y * w= y' w^{-1}$ for some $y' \leq y$ with $y' w^{-1}$ length-additive. This can be seen by computing the Demazure product $y*w^{-1}$ from right-to-left. We have that $y' \leq y \leq z \leq y'w^{-1}$, so $[y,z] \subset [y', y'w^{-1}]$ and thus $[y,z] \in \calP(w)$ as desired. For the containment $\subset$, if $[a,b] \subset [u, uw^{-1}]$ where $uw^{-1}$ is length-additive, then $a*w^{-1} \geq uw^{-1}$. This is because $a*w^{-1}$ is the Bruhat-maximal element of $\{a'w^{-1}: a' \leq a\}$, and this set contains $uw^{-1}$. So we have $b \leq uw^{-1} \leq a*w^{-1}$ as desired.
\end{proof}

The interior and exterior faces have the following explicit characterization.

\begin{proposition}
	Consider $[y,z] \in \calP(w^{-1})$, and say $y * w = vw$ where $vw$ is length-additive. Then $[y,z]$ is exterior if and only if either of the following hold
	\begin{enumerate}
		\item $yw$ is not length-additive 
		\item the leftmost subexpression (equivalently, some subexpression) for $z$ in $\bv \bw$ does not use every letter of $w$.
	\end{enumerate}
	
	Equivalently, $X_y^z$ is interior if and only if $yw$ is length-additive, and the leftmost expression (equivalently, every subexpression) for $z$ in $\mathbf{y} \bw$ uses every letter of $\bw$.
\end{proposition}

\begin{proof}
There are two kinds of ridges $X_p^q$: 
\begin{itemize}
	\item (Type I) $q=uw$ where $uw$ is length-additive and $u \lessdot p$
	\item(Type II) $pw$ is length-additive and $q \lessdot pw$.
\end{itemize} 
	We first show that condition (2) characterizes when $X^y_z$ is contained in a Type II exterior ridge.

	A Type II ridge is exterior if and only if $q \neq xw$ where $xw$ is length-additive. We claim that if the leftmost subexpression for $z$ in $\bv \bw$ does not use every letter of $\bw$, then $X_y^z$ is in a Type II ridge of $X_v^{vw}$. Indeed, consider the sequence of words obtained by taking the leftmost subexpression for $z$ in $\bv \bw$ and, from left to right, inserting the ``missing" letters from $\bv \bw$ in the appropriate places. By \cite[Lemma 2.2.1]{BB05}, each word in this sequence is reduced and, by construction, the corresponding Weyl group elements form a saturated chain from $z$ to $vw$. Say $q$ is the element of this chain covered by $vw$. By construction, the leftmost subexpression for $q$ in $\bv \bw$ does not use all letters of $\bw$. It is not hard to show that the leftmost subexpression for $q$ in $\bv \bw$ is the only subexpression for $q$ in $\bv\bw$ because $\bv \bw$ is reduced. Thus, $X_v^q$ is a Type II ridge and contains $X_y^z$ by construction.
	
	For the converse, suppose $X_y^z$ contained in a Type II exterior ridge $X_p^q$. Then we have $z \leq q \lessdot pw$ and the subexpression for $q$ in $\bp \bw$ does not use all letters of $\bw$. Now, recall that $v$ is the Bruhat maximal element of $\{p: pw~\text{length-add. and }X_y^z \subset X_p^{pw}\}$. So since $p \leq v$, there is a subexpression for $q$ in $\bv \bw$ that does not use all letters of $\bw$, and thus also such a subexpression for $z$.

	Next, we show that condition (1) is sufficient for $X^y_z$ to be contained in a Type I exterior ridge, completing the ``if" direction of the proposition. A Type I ridge is exterior if and only if $pw$ is not length-additive. We claim that if $yw$ is not length-additive, then $X_y^z$ is in an exterior Type I ridge. Indeed, for any $y' \in [v,y]$, we have $v * w \leq  y' * w \leq y*w = vw = v*w$, so $y'*w = v*w$. In particular, $\ell(y' * w) < \ell(y') + \ell(w)$. Since $y' w \leq y' * w$, we have that $y' w$ is not length-additive for all $y' \in [v, y]$. The interval $[v,y]$ must contain some $y'$ which covers $v$, and for this $y'$, $X_{y'}^{vw}$ is an exterior Type I ridge and contains $X_y^z$ by construction.
	
	Finally, for the ``only if" direction, we argue that if (1) and (2) both fail, then $X^y_z$ is interior. We argue by induction on codimension. Suppose that $yw$ is length-additive and every subexpression for $z$ in $\mathbf{y}\bw$ uses every letter of $\bw$. The base case is when $X_y^z$ is a ridge, in which case it is interior as it it contained in faces $X_y^{yw}$ and $X_{zw^{-1}}^z$. 
	
	Now suppose $X_y^z$ has codimension greater than one. Any type II ridge containing this face must be interior since the face fails (2). Suppose the face is contained in the exterior type I ridge $X_p^{vw} \subset X_v^{vw}$; we will arrive at a contradiction. We claim that $vw=z$. Indeed, note that since $y * w = yw$, $yw$ is the unique Bruhat maximal element of $\{vw: v \leq y\}$, so $z \leq vw \leq yw$.
	But all $z' \in [z, yw]$ also satisfy that every subexpression for $z'$ in $\mathbf{y}\bw$ uses every letter of $\bw$ (otherwise this property would not hold for $z$). So the faces $X_y^{z'}$ for $z' \in (z, vw]$ are, by induction, not contained in any exterior ridge but on the other hand are contained in $X_p^{vw}$. This forces $(z, vw]$ to be empty, that is, $z=vw$.
	
	Now, $p$ satisfies $v \lessdot p \leq y$. If $p*w \neq pw$, then $p*w=vw$. Also, by doing the Demazure product left to right, we see that $\mathbf{p} \bw$ has a subexpression for $vw=z$ which uses all letters of $\mathbf{p}$ and not use all letters of $\bw$. But this implies that $\mathbf{y} \bw$ also has a subexpression for $z$ which does not use all letter of $\bw$, contradicting our assumptions.	
\end{proof}

\section{Polytopal subdivisions from degenerations} \label{sec:subdiv-from-degen}

Let $T \iso {\Gm}^n$ be a torus. We distinguish four increasingly
special kinds of $T$-varieties $X$ (where ``variety'' includes ``irreducible''):
\begin{itemize}
\item those with finitely many $T$-orbits, in particular one open dense orbit,
\item those with an open dense $T$-orbit, with $X$ normal, which we
  call \defn{toric varieties},
\item normal varieties with an open dense $T$-orbit, where $T$'s generic
  stabilizer $Stab_T(X)$ is connected, and
\item normal varieties with an open dense $T$-orbit, where $T$'s generic
  stabilizer is trivial, which we call \defn{toric for $T$}.
\end{itemize}
Note that if $X$ is toric but not toric for $T$, it will anyway be toric%
\footnote{An algebraic group that is affine, connected, abelian, and
  reductive must be a torus. These properties all descend to
  quotients, so quotient groups of tori are tori.}
for $T/Stab_T(X)$.

Recall that a \defn{semitoric degeneration} $Z_1 \rightsquigarrow Z_0$
is one to a (usually reducible) $T$-scheme $Z_0$ with finitely many $T$-orbits.
In particular, the normalization of each component $X_i \subseteq Z_0$
is a toric variety.
Let $T^* := \Hom(T,\Gm)$ denote the \defn{weight lattice} of $T$,
and $\lie t_\RR^* := \RR \tensor_\ZZ T^*$ its realification.
The one-dimensional representation of $T$ with weight $\lambda \in T^*$
will be denoted $\FF_\lambda$ (where $\FF$ is our algebraically closed
base field), and we will write $\wt(T\actson \FF_\lambda) = \lambda$.
Given a projective scheme $X$ and a $T$-equivariant ample line bundle
$\calL$ on $X$, we {\em define} the \defn{moment polytope}\footnote{%
  One might be momentarily worried that when $X^T$ is infinite, this appears to be
  the convex hull of an infinite set, hence not necessarily a polytope.
  However, the function $X^T \to T^*$, $p \mapsto \wt(T \actson \calL_p)$
  is constant on connected components, of which there are finitely many.}
$\Phi(X,\calL) \subseteq \lie t_\RR^*$ as the convex hull of the $T$-weights
$\{\wt(T \actson \calL_p) \colon p\in X^T\}$;
sometimes we write $\Phi_T(X,\calL)$ if there are multiple tori in play.

When $X$ is toric with connected generic stabilizer, the moment polytope
determines $(X,\calL)$ up to $T$-equivariant isomorphism. More specifically,
if $P \subseteq \lie t_\RR^*$ is a lattice polytope, and $X$ is defined
over a field $\FF$, then $X$ can be reconstructed as
${\mathrm {Proj}}\
\FF\left[(T^* \times \ZZ) \cap \overline{\RR_+\cdot (P \times \{1\})}\right]$,
where the monoid algebra is graded (for taking $\mathrm{Proj}$) by the last
coordinate. Our definition of $\Phi(X,\calL)$ is a little silly if $X$ is
reducible -- it would be more natural to take the union of the moment polytopes
of the components -- but this will not be an important detail for us.

This section is devoted to the following theorem:

\begin{theorem}\label{thm:subdivision-general}
  Let $T$ be a torus, let $Y$ be a projective $T$-variety (probably not toric),
  and let $X \subseteq Y$ be $T$-invariant, irreducible, and with
  finitely many $T$-orbits.  Let $\calF$ be a $T$-invariant
  embedded (inside $Y$) degeneration $X \rightsquigarrow \Union X_i$,
  where the $X_i$ are the irreducible components of the special fiber,
  possibly with multiplicities $(m_i)$. 
  
  Then each $X_i$ has an open dense $T$-orbit, so the degeneration is semitoric.
  The open $T$-orbit on $X_i$ may have a larger pointwise stabilizer
  $Stab_T(X_i)$ than $X$ does, but $Stab_T(X_i)/Stab_T(X)$ is
  finite of order $m_i$.

  Let $\calL$ be an ample line bundle on $Y$, hence on $X$ and each $X_i$.
  Then the polytopes $\Phi(X_i,\calL)$ are the maximal cells of a
  subdivision of $\Phi(X,\calL)$. The map $Z \mapsto \Phi(Z,\calL)$
  is a poset isomorphism from the poset of $T$-orbit closures to
  the face poset of the subdivision.
\end{theorem}

Many of these results are well-known \cite{sturmfels1991grobner}
in the case that the degeneration is Gr\"obner, but the degenerations
in Theorem \ref{thm:Hclass} are not of this type so we needed to prove
the result more generally.

\subsection{Lemmata on moment polytopes}

We begin with a couple of properties of moment polytopes, before we get
to a crucial tool, Duistermaat--Heckman measures.
Since these properties are usually proven in the symplectic setting
(under a very different-looking definition of the measure), we include
complete proofs. This is the setting also of \cite{brion1990action}
and \cite{dolgachev1998variation}, which we will make use of,
but their results are more focused on GIT quotients and less on properties of
the polytopes and measures. In particular we claim no originality of
the remaining results in this section, just their combination in
the theorem above.

\begin{lemma}\label{lem:eqvtLineBundles}
  Let $\Gm$ act on $\PP^1$ such that the tangent space $T_0\PP^1$ has
  $\Gm$-weight $w$.
  Let $\calL$ be a $\Gm$-equivariant line bundle on $\PP^1$, of degree $d$.
  Identifying $\Gm^* \iso \ZZ$ as usual, we have
  $\wt(\Gm\actson \calL_\infty)-\wt(\Gm\actson \calL_0) = d\cdot w$.
  In particular, if $\calL$ is ample and the action is standard $(w=1)$,
  then $\wt(\Gm\actson \calL_\infty) > \wt(\Gm\actson \calL_0)$.

  More generally, let $T$ act on $\PP^1$ fixing $0,\infty$ and
  on a $T$-equivariant line bundle $\calL$ of degree $d$.
  Then 
  $\wt(T\actson \calL_\infty)
  = \wt(T\actson \calL_0) + d\cdot \wt(T\actson T_0 \PP^1)$, and for $d\geq 0$
  the $T$-weights occurring in $\Gamma(\PP^1, \calL)$ are
  $\wt(T\actson \calL_0) + i\cdot \wt(T\actson T_0 \PP^1)$, $i \in [0,d]$.
\end{lemma}

\begin{proof}
  Our $\PP^1$ is equivariantly isomorphic to $\PP(\FF_0\oplus \FF_w)$,
  and the weights $\Gm \actson \calO(1)_{0,\infty}$ are $0,w$ respectively.
  Every equivariant line bundle on $\PP^1$ is of the form
  $\calO(d) \tensor \FF_j$ where $\FF_j$ is the $\Gm$-representation
  of weight $j$ (giving a trivial, but equivariantly nontrivial, line bundle);
  on such an $\calL$ the weights $\Gm \actson \calL_{0,\infty}$
  are $j,j+d\cdot w$ respectively.

  The $T$-version follows from the same identification of the line bundle.
\end{proof}

\begin{lemma}\label{lem:fixedPoints}
  \begin{enumerate}
  \item On each component $Z \subseteq X^T$, the weight
    $\wt(T\actson \calL_y)$ is constant in $y$.
  \item The fixed-point components are all connected to one another
    by $T$-invariant rational curves. More specifically: given two $T$-fixed
    points, one can get from one to the other by moving within a component
    of $X^T$, or following a $T$-fixed rational curve, and repeating
    these moves.
  \end{enumerate}
\end{lemma}

\begin{proof}
  (1) The weight varies continuously, inside a discrete set $T^*$.

  (2) Bertini's theorem implies that any two points $p,q$ on $X$ can
  be connected by {\em some} curve $\gamma$. Since $X$ is projective,
  $\gamma$ defines a point in some (also projective)
  Hilbert scheme, and (now taking $p,q\in X^T$) each $t\cdot \gamma$ also
  connects $p,q$. Taking a limit (i.e. picking
  $\gamma' \in \overline{T\cdot\gamma}^T$)
  we get a sequence of $T$-invariant curves connecting $p,q$, some of which are
  $T$-fixed pointwise (hence lie within components of $X^T$) and some of
  which aren't (in which case they must be rational curves, as higher genus
  curves don't admit nontrivial torus actions). 
\end{proof}

\begin{proposition}\label{prop:spanOfPhi}
  Let $T \actson (X,\calL)$ with $X$ a projective variety.
  Let $S \leq T$ be the kernel of the action on $X$, i.e. the generic
  stabilizer. Then $\Phi_T(X,\calL)$ is contained in a translate of
  the subspace $S^\perp := \RR \tensor \{\lambda \in T^*\colon \lambda(S)=1\}
  \leq \lie t^*_\RR$, and (for $\calL$ ample) in no smaller affine subspace.
\end{proposition}

\begin{proof}
  The reason we might face a translate,
  rather than seeing $\Phi_T(X,\calL)$ live directly in $S^\perp$,  
  is that $S$ may act nontrivially on $\calL$. (However, $S$ at least
  must act with the same weight $\wt(S \actson \calL_p) \in S^*$ for
  every $p\in X$, as in Lemma \ref{lem:fixedPoints}(1).)
  ``Fix'' this by tensoring $\calL$ with $\FF_{-\lambda}$,
  where $\lambda$ is the weight $\wt(T\actson \calL_p)$ for $p$ an
  arbitrarily chosen $T$-fixed point.  It is trivial to check that
  $\Phi_T(X,\calL) = \Phi_T(X,\calL\tensor \FF_{-\lambda}) + \lambda$.
  Having reduced to the case of $S$ acting trivially on each fiber of $\calL$,
  we have the claimed containment in $S^\perp$.

  For the second claim, start by translating $\Phi_T(X,\calL)$ by $-\lambda$
  as above, so the statement is that the linear span of $\Phi_T(X,\calL)$
  is all of $S^\perp$. If not, that span is contained in a proper
  subspace $R^\perp$ for a larger subgroup $R > S$; such an $R$ acts on
  each $\calL_p, p\in X^T$ trivially but doesn't fix $\calL$ pointwise.

  Since $R$ properly contains $X$'s generic stabilizer $S$, there is
  some one-parameter group $Q \into R$ and some point $x\in X$ such
  that $x$ is not invariant under $Q$. Then the map
  $\Gm \to X$, $z\mapsto Q(z)\cdot x$ extends to a $Q$-equivariant
  {\em nonconstant} map
  $\gamma\colon \PP^1 \to X$, so $\gamma^*(\calL)$ is therefore
  of positive degree (this is where we use $\calL$ ample).
  Now apply Lemma \ref{lem:fixedPoints} to show that the $Q$-weights on
  the $0,\infty$ fibers of $\gamma^*(\calL)$ are different, contradicting
  the condition that they had the trivial $R$-action.
\end{proof}

One nice aspect of the polytope picture of projective toric varieties,
compared to the fan picture, is that the correspondence between faces and
$T$-invariant subvarieties is inclusion-preserving rather than
inclusion-reversing. This has a partial extension (well-known to 
Hamiltonian geometers) to general actions
$T \actson (X,\calL)$ for $X$ a projective variety.
Call a subvariety $Y\subseteq X$ \defn{attractive}
if $Y$ is the \defn{(Bia\l ynicki-Birula) sink} for some $S\colon \Gm\to T$,
i.e. if $Y$ is a component of $X^S$ and the set
$ \{x\in X\colon \lim_{z\to 0} S(z)\cdot x \in Y\} $ is open dense in $X$.
(In particular $Y$ is necessarily $T$-invariant.)
Equivalently, $Y$ is a component of $X^S$ for some coweight
$S\colon \Gm\to T$, such that the $S$-weights
on each Zariski normal space $T_p X/T_pY$ are nonnegative.
In the case that $X$ is toric, every $T$-orbit closure is attractive.

\begin{proposition}\label{prop:faceCorrespondence}
  If $Y \subseteq X$ is attractive in the above sense, then $\Phi(Y,\calL)$
  is a face of the polytope $\Phi(X,\calL)$, and this gives a
  correspondence between the poset of attractive subvarieties
  and the face poset of $\Phi(X,\calL)$.
  The inverse map, taking a face $F$ of $\Phi_T(X,\calL)$ to an attractive
  subvariety $Y$, can be described in three equivalent ways:
  \begin{enumerate}
  \item Pick a coweight $S$ defining a linear functional on $\Phi_T(X,\calL)$
    minimized on exactly $F$. Then let $Y$ be the sink for $S$.
  \item Let $Y = \{x\in X\colon \Phi \left(\overline{T\cdot x},\calL \right)
    \subseteq F\}$.
  \item Define a homogeneous ideal
    $I := \Oplus_n \Oplus_{\lambda \in \Phi_T(X,\calL)\setminus F}
    \Gamma(X;\, \calL^{\tensor n})^{n\lambda\text{ weight space}}$
    inside the \defn{section ring} $\Oplus_n \Gamma(X;\, \calL^{\tensor n})$
    for $X$. Then $Y = V(I)$. 
  \end{enumerate}
  \junk{In the toric case $\dim Y = \dim_\RR \Phi(Y,\calL)$, but in general
    we only have $\geq$.}
\end{proposition}

In a rather failed attempt to avoid including an entire treatise on algebraic
Hamiltonian geometry we give the proof only of the part we need later:
that given a face $F$, there {\em exists} a subvariety $Y$ with
$\Phi(Y,\calL) = F$.

\begin{proof}[Proof that (3) defines such a $Y$\!]
  \junk{
  First we need to show that $\Phi(Y,\calL)$ is a face $F$ in the usual sense,
  i.e. there is a linear functional on $\Phi(X, \calL)$ minimized
  on exactly $\Phi(Y, \calL)$. Pick a coweight $S$ defining $Y$, and
  let $\sigma$ be the corresponding functional on $\lie t^*_\RR$
  Then since $Y$ is $S$-fixed pointwise, the functional $\sigma$ is
  constant on $\Phi(Y,\calL)$ by Proposition \ref{prop:spanOfPhi}.
  By the analysis in the proof of Lemma \ref{lem:fixedPoints}(2),
  a component $Z$ of $X^S$ is the sink iff the $S$-weight $\wt(S\actson\calL_p)$,
  $p\in X^S$ is minimized on $Z$. 

  Given a face $F \subseteq \Phi_T(X,\calL)$, there is a rational functional
  $\lie t^*_\RR$ minimized exactly on $F$, i.e. a rational coweight.
  Some positive multiple of it defines an actual coweight $S$.
  }

  First we prove that $I$ is a prime ideal. Since it is multigraded
  (by degree and $T$-weight), we need to prove that if $i\in I, r\in R$ are
  multihomogeneous, then $ri \in I$. Let $F \subseteq \Phi(X,\calL)$ be
  defined by the linear inequality $f\geq c$, and consider the linear
  functional $(-c,f)$ on the space of pairs (degree,weight).
  This functional is nonnegative on the pairs (degree,weight) in $R$
  and is strictly positive on those in $I$, hence positive on
  the (degree,weight) of $ri$. Therefore $ri \in I$.

  Second we prove that $Y := V(I)$ contains enough fixed points to
  hit the vertices of $F$. Let $\lambda$ be a vertex of $F$; since it
  is itself a face we can construct $Z := V(J)$ just as we did $Y$,
  and obviously $J\geq I$ so $V(J) \subseteq V(I)$. 
  In this way we reduce to the case that $F$ is a vertex.

  Pick a point $p\in X^T$ with $\wt(T\actson \calL_p) = \lambda$,
  and let $K\leq R$ be the radical ideal defining $p$.
  To show $p \in V(J)$, it suffices to show $K \geq J$,
  i.e. that every section $\sigma \in \Gamma(X;\, \calL^{\tensor n})$
  with a well-defined weight $\mu \neq n\lambda$ vanishes at $p$.
  But this is trivial; $\sigma|_p$ is a weight $\mu$ element of
  $\calL^{\tensor n}_p$, a one-dimensional representation with weight $n\lambda$,
  hence $\sigma|_p$ must be zero.
\end{proof}

\begin{lemma}\label{lem:sink}
  Let $S\colon \Gm\to Aut(X,\calL)$ be a one-dimensional group acting on a
  projective variety and ample line bundle, so $\Phi_S(X,\calL) \subseteq \RR$.
  For $z$ in the Bia\l ynicki-Birula sink, 
  we have $wt(S\actson \calL_z) = \min\Phi_S(X,\calL)$.
\end{lemma}

\begin{proof}
  Let $Y$ be the component of $X^S$ containing $y$. Then by \cite{Konarski},
  we have $\dim X_Y^\circ + \dim X^Y_\circ \geq \dim Y + \dim X$.
  (Warning: the $X^\pm$ loci defined in \cite{Konarski} are not quite
  $X_Y^\circ,X^Y_\circ$, so his inequality looks a little different.)
  Consequently, if $Y$ is not the sink (with $\dim X_Y^\circ = \dim X$),
  then $X^Y_\circ$ properly contains $Y$.

  By its definition, $\Phi_S(X,\calL) \ni wt(S\actson \calL_z)$,
  so we need to show for any $y\in X^S$,
  that $wt(S\actson \calL_x) \geq wt(S\actson \calL_z)$.
  If $Y$ is the sink, we are done. Otherwise by the above,
  we can pick $x \in X^Y_\circ \setminus Y$. Use this point to define a
  {\em nonconstant} $S$-equivariant map
  $\gamma\colon \PP^1 \to X$, $t \mapsto S(t)\cdot x$
  for $t\in\Gm$ (using $X$ projective); the nonconstancy implies that
  the degree of $\gamma^*(\calL)$ is strictly positive.
  Using Lemma \ref{lem:eqvtLineBundles}
  we learn $wt(S\actson \calL_{\gamma(0)}) < wt(S\actson \calL_{\gamma(\infty)})$,
  and in particular, that $\gamma(0)$ is in a different component of $X^S$.
  Chain these inequalities together finitely many times (as there are
  only finitely many components of $X^S$), getting stuck at the sink.
\end{proof}

\begin{proposition}\label{prop:TorbitClosure}
  If $X$ is irreducible, and $x\in X$ is general, then
  $\Phi\left(\overline{T\cdot x}, \calL\right) = \Phi(X, \calL)$.
  If $T$ acts faithfully on $X$, and $x$ is even more general,
  then the normalization of $\overline{T\cdot x}$ is the
  toric variety associated to the polytope $\Phi(X, \calL)$.
\end{proposition}

\begin{proof}
  For each vertex $v$ of $\Phi(X, \calL)$, pick a coweight $S_v$
  whose minimum value on $\Phi(X, \calL)$ is achieved on $v$ (and only there).
  Each of these circle actions defines an open Bia\l ynicki-Birula
  stratum in $X$, with sink $Z_v$; by $X$'s irreducibility these
  finitely many open sets have nonempty intersection. Take $x$ in this
  intersection. By definition, $\lim_{t\to 0} S_v(t)\cdot x \in Z_v$, and
  $\overline{T\cdot x} \supseteq \overline{S_v\cdot x}$,
  so $\overline{T\cdot x}$ intersects $Z_v$ in some point $p_v$.
  By Lemma \ref{lem:sink} we have $\wt(T\actson \calL_{p_v}) = v$,
  so $\Phi\left(\overline{T\cdot x},\calL\right) \supseteq \{v\}$.

  For the second statement, we need $x$ also inside the open set where
  $T$ acts with trivial stabilizer. Normalization is functorial,
  so $T$ acts on the normalization of $\overline{T \cdot x}$.
  The normalization map
  is an isomorphism over an open set, necessarily $T$-invariant,
  and the unique smallest $T$-invariant open set in $\overline{T\cdot x}$
  is the open orbit. Hence $T$ acts with a free open orbit on the
  normalization as well, making it a toric variety. It is then
  straightforward to check that it is the toric variety associated
  to the polytope $\Phi(X,\calL)$.
\end{proof}

\subsection{Dissections and subdivisions}

In the literature, one can find many references to a philosophy that
semitoric degenerations of toric varieties should
correspond to {\em regular subdivisions} of their moment polytopes, where each cell in the subdivision is the moment polytope of a component in the special fiber.
It is true that regular subdivisions give rise to semitoric degenerations
\cite{sturmfels1991grobner}. One has to be more careful
with the converse: Gr\"obner degenerations (which, if $X$ is toric,
are automatically semitoric) do give rise to regular subdivisions \cite{sturmfels1991grobner}. If the semitoric degeneration is not necessarily Gr\"obner, then, as we will show, we still obtain a subdivision but it is not \emph{a priori} regular.
\junk{
\AK{Should we attempt to demonstrate a degeneration to a nonregular
  subdivision? There's $\PP^2$ inside $\PP(Sym^4\CC^3)$ to a union of $7$
  triangles. (I'm fine with creating new vertices.) Is there a simpler
  example? The pieces don't need to be triangles.} \MSB{I think that would be nice, if it's not too terrible. The 7-triangles example is the smallest/simplest nonregular subdivision I know of.}
}

\begin{definition}
  Let $P$ be a $d$-dimensional polytope. A collection $\{P_i\}$ of
  $d$-dimensional polytopes forms a \defn{dissection} of $P$ if
  $\bigcup P_i = P$ and the polytopes $P_i$ have pairwise disjoint
  interiors. (Equivalently, Lebesgue measure on $P$ is the sum of
  the Lebesgue measures on the $P_i$.) We call the polytopes $P_i$ the \defn{cells} of the subdivision. A dissection is a \defn{subdivision} 
  if any intersection $P_i \cap P_j$ is a face of each (or empty).
  It is a \defn{regular subdivision} if it is induced from a
  piecewise-linear continuous convex function on $P$,
  where the $P_i$ are the domains of linearity.
\end{definition}

\begin{figure}[h]
  \centering
\begin{tikzpicture}
    \draw [thick] (0,0) -- (4,0) -- (4,4) -- (0,4) -- cycle;
    
    \draw (2,0) -- (2,4);
    
    \draw (2,2) -- (4,2);
  \end{tikzpicture}
  \qquad
\begin{tikzpicture}
    \coordinate (A) at (0,0);
    \coordinate (B) at (5,0);
    \coordinate (C) at (2.5,4);

    \draw [thick] (A) -- (B) -- (C) -- cycle;

    \coordinate (D) at (1.5,1);
    \coordinate (E) at (3.5,1);
    \coordinate (F) at (2.5,2.5);

    \draw (D) -- (E) -- (F) -- cycle;
    \draw (D) -- (B) -- (E);
    \draw (E) -- (C) -- (F);
    \draw (F) -- (A) -- (D);
\end{tikzpicture}
  \caption{A dissection that isn't a subdivision, and a subdivision that isn't regular.}
  \label{fig:dissections}
\end{figure}
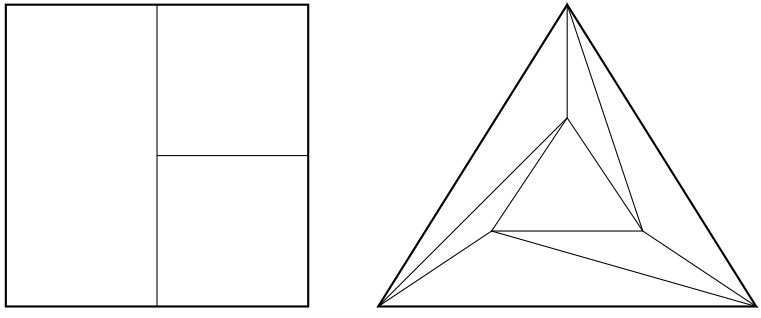

\subsection{Duistermaat--Heckman measures}

Given a $T$-equivariant ample line bundle $\calL$ on a projective $T$-scheme $X$
over a field $\FF$, following \cite{brion1990action,guillemin1996symplectic}
define\footnote{%
  Duistermaat and Heckman defined these in the symplectic manifold setting rather than
  algebraic variety setting, which is less general (as it disallows singularities
  or base fields other than $\CC$) and more general (as some symplectic manifolds
  can't be made algebraic). We attempted to write proofs in the symplectic
  language but getting it to deal with the potential singularities that arise
  under degeneration was extremely cumbersome.}
the \defn{Duistermaat--Heckman measure} on $\lie t_\RR^*$
as a weak limit of Dirac measures:
\begin{equation}\label{eq:DH-def}
	DH(X,\calL) := \lim_{n\to\infty} \frac{1}{n^{\dim X}}
	\sum_{\lambda \in \frac{1}{n}T^*}
	\dim_\FF(n\lambda\text{ weight space in }\Gamma(X;\ \calL^{\tensor n}))
	\, \delta_\lambda
\end{equation}

If there are multiple tori at play, we sometimes write $DH^T(X, \calL)$
to include the torus in the notation.

\begin{lemma}\label{lem:multDiagram}
  Assuming $X$ is reduced,
  the sum in \eqref{eq:DH-def} (even before taking the limit) can be restricted to
  $\lambda \in \Phi(X,\calL)$, as the other terms vanish.
\end{lemma}

\begin{proof}
  If $\lambda \notin \Phi(X,\calL)$, there exists a coweight
  $R\colon \Gm\to T$ such that
  $\langle R, \lambda\rangle < \langle R, p\rangle$ for all
  $p\in \Phi(X,\calL)$.  Let $\overline\lambda$ denote the projection
  of $\lambda$ under the map $\lie t^*_\RR \to \lie r^*_\RR$.
  
  Let $\sigma$ be a section of $\calL^{\tensor n}$ of $R$-weight
  $n\overline\lambda$, and $x\in X$. 
  Define $\gamma\colon \PP^1 \to X$, $z\mapsto R(z)\cdot x$ for
  $z\in \Gm$, as in the proof of Lemma \ref{lem:fixedPoints}.
  Then applying the last part of Lemma \ref{lem:eqvtLineBundles}
  to $\gamma^*(\calL^{\tensor n})$, we see that $n\overline\lambda$
  is not an $R$-weight in $\Gamma(\PP^1,\gamma^*\calL^{\tensor n})$.
  Hence $\sigma$ vanishes at $x$, and $x$ was arbitrary.
  Since $X$ is reduced, and $\sigma$ vanishes at the points of $X$,
  we have $\sigma=0$.
\end{proof}

\begin{proposition}\label{prop:DH}
  \begin{enumerate}
  \item If $X$ (or its normalization) is the toric variety associated
    with a polytope $P$, then $\Phi(X,\calL) = P$ and $DH(X,\calL)$ is
    Lebesgue measure thereupon.  \junk{ \item If $X$ is irreducible,
      then the support of $DH(X,\calL)$ is the moment polytope $\Phi(X,\calL)$.}
  \item If $X$ is reduced, the measure $DH(X,\calL)$ is supported
    on\footnote{Much more is true, though we won't prove it:
    	if $X$ is irreducible, not necessarily reduced, then the support
    	is {\em exactly} $\Phi(X,\calL)$.} $\Phi(X,\calL)$. 
  \item Say $h\colon S \to T$ is a homomorphism of tori,
    and $\wh h\colon\lie t^*_\RR \to \lie s^*_\RR$ the $\RR$-linear
    extension of the transpose map on weight lattices. 
    Then $\Phi_S(X,\calL) = \wh h(\Phi_T(X,\calL))$, and
    $DH^S(X,\calL) = \wh h_*(DH^T(Y,\calL))$ using the natural
    pushforward of measures.

    If $S=1$ this shows that the total measure of $DH(X,\calL)$ is the
    leading coefficient of $X$'s Hilbert polynomial, hence strictly positive.
    \junk{
    \item If $X$ is irreducible, and $T$ acts with generic stabilizer $S$,
      then $\dim_\RR \Phi(X,\calL) = \dim(T/S)$. More specifically,
      $\Phi(X,\calL)$ is contained in a translate of $S^\perp$ (as in
      Proposition \ref{prop:spanOfPhi}).
    \item If $T$ acts on $Y$ with a dense orbit, and the generic stabilizer has
      $k$ components, then $DH(Y,\calL)$ is $\frac 1 k$ times Lebesgue measure
      on the moment polytope $\Phi(Y,\calL)$ (even if it is not of full dimension).
    }
  \item If $X$ is reducible, so $X = \Union X_i$ and the $X_i$ have
    multiplicities $m_i$ and reductions $(X_i)_{red}$, then
    $DH(X,\calL) = \sum m_i DH((X_i)_{red},\calL)$ where the sum is over the
    top-dimensional (in particular, geometric) components of $X$.
  \item If $X \subset Y$ where $Y$ is projective, $X \rightsquigarrow \Union X_i$ is an embedded
    $T$-invariant degeneration, and $\calL$ is a
    $T$-equivariant ample line bundle on $Y$ (hence on $X$ and each
    $X_i$), then $DH(X,\calL) = DH(\Union X_i,\calL)$.

    If $X$ has a dense $T$-orbit, with trivial (or finite)
    generic stabilizer, then $T$ acts on each $X_i$ with finite
    generic stabilizer $\Gamma_i$ (making $X_i$'s normalization
    a toric variety for the action of $T/\Gamma_i$).
  \item If $X_1,X_2 \subseteq Y$ are subvarieties with
    $\Phi(X_1,\calL)$, $\Phi(X_2,\calL)$ disjoint, then $X_1,X_2$ are disjoint.
  \end{enumerate}
\end{proposition}

\begin{proof}
  \begin{enumerate}
  \item If $X$ is the toric variety, the construction ${\mathrm {Proj}}\
   \FF\left[(T^* \times \ZZ) \cap \overline{\RR_+\cdot (P \times \{1\})}\right]$
    of $X$ lets
    one see that the weights in $\Gamma(X;\ \calL^{\tensor n})$ are exactly those
    in $nP \cap T^*$, each with multiplicity $1$. Dividing the position by $n$,
    we get the points in $P \cap \frac{1}{n}T^*$, the number of which grows like
    $n^{\dim_\RR P} = n^{\dim X}$. This visibly limits to Lebesgue measure.

    If $\nu\colon Y\to X$ is the normalization, and $Y$ is the toric variety,
    then $X$'s section ring $R'$ is a subring of $Y$'s section ring $R$.
    By the $T$-equivariance, the conductor $f(R/R')$ is multihomogeneous.
    Pick a multihomogeneous element $g$ in the conductor,
    so $R_n \geq R'_n \geq g R_{n-\deg g}$. If we generalize the definition
    of $DH$ to $\NN\times T^*$-graded vector spaces,
    \[ DH(V) :=  \lim_{n\to\infty} \frac{1}{n^{\dim X}}
    \sum_{\lambda \in \frac{1}{n}T^*}
    \dim_\FF(n\lambda\text{ weight space in }V_n)    \, \delta_\lambda
    \]
    we get $DH(R) \geq DH(R') \geq DH(gR)$, and it is easy to see
    that $DH(gR) = DH(R)$. Hence $DH(X,\calL) = DH(Y,\calL)=$ Lebesgue measure
    on $P$.
  \item This follows easily from Lemma \ref{lem:multDiagram} and the definition.
  \item $\Phi_S(X,\calL)$ is the convex hull of $\{\wt(S\actson \calL_p), p\in X^S\}$.
    Two things are true about each connected component $F \subseteq X^S$:
    the function $p\mapsto \wt(S\actson \calL_p)$ is constant on $F$,
    and $T$ preserves $F$. By Borel's fixed point theorem, $F^T \neq \emptyset$.
    Hence $\Phi_S(X,\calL)$ is the convex hull of $\{\wt(S\actson \calL_p), p\in X^T\}$.
    Then use the fact that the linear projection of the convex hull of a point set 
    is the convex hull of the linear projections.
  \item This is not very difficult, but to save space we quote
    \cite[Theorem 9.8]{baumann2021mirkovic}.
  \item \label{item:Kproof}
    Using Proposition \ref{prop:flatness} we know
    that $[X] = [\Union X_i]$ in $K_0^T(Y)$. Regard $K_0^T(pt)$ as formal
    differences of $T$-representations, and let $\pi_W\colon W\to pt$
    denote the projection. Then for $Z_1 = X$ or $Z_0 = \Union X_i$ we have
    \begin{eqnarray*}
      [\Gamma(Z_s;\, \calL^{\tensor n})]
      &=& \left[ \sum_i (-1)^i H^i(Z_s;\, \calL|_{Z_s}^{\tensor n})\right]
          \qquad\text{for $n\gg 0$, by Serre vanishing} \\
      &=& (\pi_{Z_s})_* [\calL|_{Z_s}^{\tensor n}]
          \qquad\text{by definition of pushforward in $K$-theory} \\
      &=& (\pi_Y)_*\left( [\calO_{Z_i}]\, [\calL^{\tensor n}] \right)
          \qquad\text{by the push-pull formula} 
    \end{eqnarray*}
    as elements of $K_0^T(pt)$. Since $[\calO_{Z_0}] = [\calO_{Z_1}]$,
    we can chain the two sequences of equations together to get
    $[\Gamma(X;\, \calL^{\tensor n})] = [\Gamma(\Union X_i;\, \calL^{\tensor n})]$,
    or equivalently,
    $\Gamma(X;\, \calL^{\tensor n}) \iso \Gamma(\Union X_i;\, \calL^{\tensor n})$
    as $T$-representations. So the two Duistermaat--Heckman measures
    are computed from the same thing.

    For the second claim, note first using (1) and (2) that
    $DH(X,\calL)$ is absolutely continuous with respect to Lebesgue
    measure on $\lie t^*_\RR$, for the simple reason that it equals
    Lebesgue measure times the function that is $\frac 1 d$ on
    $\Phi(X,\calL)$ and $0$ elsewhere,
    where $d$ is the size of the generic stabilizer.

    Note second that if $T$ acts with positive-dimensional stabilizer
    $R$ on $X_i$, then $DH(X_i,\calL)$ is {\em not} absolutely
    continuous with respect to Lebesgue measure on $\lie t^*_\RR$,
    because by Proposition \ref{prop:spanOfPhi} it is supported on
    a translate of a proper linear subspace (and is not zero, by the positivity
    in (2)). Since each $DH(X_i,\calL)$ is a positive measure, and
    they add up to one that is absolutely continuous with respect to
    Lebesgue measure on $\lie t^*_\RR$, each summand individually is likewise
    absolutely continuous, precluding the positive-dimensional stabilizers.
  \item Let $R$ be the section ring for $Y$, and $I_1,I_2$ the ideals
    defining $X_1,X_2$. Then using Lemma \ref{lem:multDiagram} we see
    that $R/(I_1+I_2)$ is $0$ in all positive degrees, hence has
    empty $\Proj$.
    \qedhere
  \end{enumerate}
\end{proof}

The simplicity of the proof of (\ref{item:Kproof}) suggests that all
of these could be more naturally stated in $K$-theory, effectively
considering $n\gg 0$ but not taking the limit. This is true except for (3),
whose $K$-analogue then involves a bunch of inclusion-exclusion, and
becomes especially subtle when there are nonreduced components $X_i$.

The above proposition isn't quite enough to get us even
a dissection of the moment polytope from a degeneration.
One phenomenon that must be dealt with is seen in the degeneration
$\{([a,b,c],z)\colon b^2 = zac\} \subseteq \PP^2 \times \AA^1$, where
$t\cdot ([a,b,c],z) := ([a,tb,t^2c],z)$ for $t\in\Gm$. In this example,
a conic with a generically free action degenerates to a double line with a
global $\ZZ/2\ZZ$-stabilizer. Both moment polytopes are $[0,2]$ (w.r.t.
$\calL = \calO_{\PP^2}(1)$) and the resulting formula on
Duistermaat--Heckman measures is $L_{[0,2]} = 2\cdot\frac 1 2 L_{[0,2]}$,
where the $2$ comes from the scheminess and the $\frac 1 2$ from the stabilizer.
(This matching of multiplicity and stabilizer size is no accident; see Theorem \ref{thm:subdivision-general}.)

With that phenomenon in mind, picture (impossibly) a $\PP^1$ with faithful
$\Gm$-action degenerating to some union of {\em two} $\PP^1$s, {\em each}
carrying a $\Gm$ action with global stabilizer $\ZZ/2\ZZ$. Then writing $L_P$
for Lebesgue measure on $P$, using parts (1), (2), (4), (5) we would
get a formula $L_{[m,n]} = \frac 1 2 L_{[m,n]} + \frac 1 2 L_{[m,n]}$,
without the desired disjointness of supports. This doesn't in fact
happen, but it doesn't seem possible to exclude it using only the
proposition above.

We could at this point prove a limited dissection result based on an additional
proviso that no $X_i$ has different global stabilizer and multiplicity.
Instead, we'll exclude that possibility using a connectedness result,
which we'll need anyway in order to upgrade our dissections to subdivisions.

\subsubsection{Geometric invariant theory quotients}

Let $X = \Proj R$, where $T$ acts on $R$; for example, $R$ could be
a section ring $\Oplus_{n\in\NN} \Gamma(X;\, \calL^{\tensor n})$ for an
ample line bundle $\calL$.
To a point $\lambda \in S^*$ (or even in $\QQ\tensor S^*$), 
we associate a certain ``geometric invariant theory quotient''
\[ X//_\lambda\, T := \Proj \Oplus_{n\in\NN} (R_n)^{\text{$n\lambda$ $T$-weight space}} \]
\junk{
  carrying a $T/S$-action. (The definition doesn't depend on the larger
  torus $T$ containing $S$, but we will most often use it in the setting
  where there is a larger torus.) }
There are several things to be careful about:
\begin{itemize}
\item The choice of $\calL$ (or more generally choice of $R$) is
  traditionally suppressed in this notation, despite it being
  necessary to define the quotient.
\item While $X//_\lambda\, T$, being a $\mathrm{Proj}$, carries a rank $1$ sheaf
  we will denote $\calL // T$,
  that sheaf may not be a line bundle (although some tensor power of it
  will be). 
\item We will most often be dividing by a subgroup $S\leq T$.
  While the action of $T$ descends to an action of $T/S$ on
  $X//_\lambda\, S$, it will not descend to an action on $\calL//S$
  unless $\lambda = 0$.
\end{itemize}
As such it behooves us to extend the definition of moment polytope
to handle the case that $\calL$ is a rank $1$ sheaf such that
some $\calL^{\tensor m}, m>0$ is a line bundle. This is easy:
\[ \Phi(X,\calL) := \frac{1}{m} \Phi(X,\calL^{\tensor m}) \]
When $\calL$ is a line bundle,
it is trivial to check that this definition doesn't depend on $m$.
For the general case, if $\calL^{\tensor m_1}$ and $\calL^{\tensor m_2}$ are
line bundles, then $\calL^{\tensor \mathrm{lcm}(m_1,m_2)}$ will be too,
and the three will all define the same polytope. It is worth noting
that the vertices of $\Phi(X,\calL)$ may not be in the weight lattice $T^*$.
(Indeed, the recipe $\Proj \FF\left[(T^* \times \ZZ) \cap
  \overline{\RR_+\cdot (P \times \{1\})}\right]$ for projective toric varieties
works perfectly well for polytopes $P$ with rational vertices.)
We warn the reader that the definition of Duistermaat--Heckman measure
requires more serious reworking (try applying the definition as stated
to the case $R = \FF[x]$ where $T$ is trivial and $x$ has degree $2$;
the limit doesn't exist).

We need a couple of results from geometric invariant theory, which
we will quote from \cite{brion1990action,dolgachev1998variation}.
Following \cite{tolman1998examples},
define the \defn{x-ray} of the action $T\actson (X,\calL)$ as the set
of moment polytopes of the irreducible components of $X^S$, where $S \leq T$
varies over the (finitely many) point stabilizers.
If $S\leq T$, one can compute the x-ray for the $S$-action
by projecting each of the polytopes along $\lie t^*_\RR \onto \lie s^*_\RR$,
discarding those that become codimension $0$. 

The inclusion $ \Oplus_{n\in\NN} (R_n)^{\text{$n\lambda$ $T$-weight space}}
\into R$ induces a rational, $T$-invariant map
$X \dasharrow X//_\lambda\, T$, and we will need to know the extent to
which it is just ``divide by $T$''. Given the data $X,\calL,T,\lambda$,
call\footnote{This is a rephrasing of the Hilbert-Mumford criterion.}
a point $x\in X$ \defn{semistable} if $\lambda$ is in
$\Phi_T\left(\overline{T\cdot x},\calL\right)$,
and \defn{properly stable} if $\lambda$ lies in the interior of
$\Phi_T\left(\overline{T\cdot x},\calL\right)$.

\begin{proposition}\cite{brion1990action}\label{prop:brionprocesi}
  \begin{enumerate}
  \item The stable set $X^s$ is a $T$-invariant open set (possibly empty)
    with only finite stabilizers.
  \item The rational map $X \dasharrow X//_\lambda\, T$ is defined on
    the semistable set $X^{ss}$, giving a map $X^{ss}/T \onto X//_\lambda\, T$.
    Use the codimension $1$ pieces in the x-ray to divide
    $\Phi_T(X,\calL)$ into chambers. If $\lambda$ lies in the interior
    of a chamber, then the semistable points are all stable, 
    and the map $X^{ss}/T \to X//_\lambda\, T$ is bijective.
  \end{enumerate}
\end{proposition}

\begin{proposition}\label{prop:slicing}
  \begin{enumerate}
  \item If $X$ is a variety, then $X//_\lambda\, T$ is nonempty iff
    $\lambda \in \Phi_T(X,\calL)$, in which case $X//_\lambda\, T$ is a variety.
  \item $\Phi_T(X//_\lambda\, S, \calL//S)
    = \Phi_T(X,\calL) \cap (\lambda + \lie s^\perp)$
  \end{enumerate}
\end{proposition}

\begin{proof}
  \begin{enumerate}
  \item The first part follows from Proposition \ref{prop:brionprocesi}(1)
    and the definition of semistable;
    $\lambda \notin \Phi_T(X,\calL)$ iff $X^{ss}=\emptyset$.
    The second derives from the simple fact that a subring (here the
    $T$-covariant subring) of a domain is a domain.    
  \item For convenience, replace $\calL$ with $\calL^{\tensor m}$
    and $\lambda$ with $m\lambda$, in order to ensure that
    $\calL^{\tensor m}//S$ is a line bundle; this scales both sides
    by $m$, so doesn't change the claim to be proven.

    Let $v$ be a vertex of $\Phi_T(X,\calL) \cap (\lambda+\lie s^\perp)$;
    we want to show it lies in $\Phi_T(X//_\lambda\, S,\calL//S)$.
    Then there is a face $F$ of $\Phi_T(X,\calL)$ such that
    $F \cap (\lambda+\lie s^\perp) = \{v\}$.
    By the part of Proposition \ref{prop:faceCorrespondence} we proved,
    there is a subvariety $Y \subseteq X$ with $\Phi_T(Y,\calL) = F$.
    By (1), $Y//_\lambda S \neq \emptyset$, so by Borel's theorem
    it must contain a $T$-fixed point $y$, who provides the desired point
    in $\Phi_T(X//_\lambda\, S,\calL//S)$.

    For the converse, we need to show that
    $\Phi_T(X//_\lambda\, S,\calL//S)$ is contained in
    $\Phi_T(X, \calL)$ and $\lambda + \lie s^\perp$. These both follow
    easily from the definitions of $X//_\lambda\, S$ and moment polytope.
    \qedhere
  \end{enumerate}
\end{proof}

\subsection{Putting it all together}

\begin{proof}[Proof of Theorem \ref{thm:subdivision-general}]
  Let $F \subseteq Y \times \PP^1$ be the $T$-invariant embedded degeneration.
  We have four things to prove:
  \begin{itemize}
  \item the $X_i$ have dense $T$-orbits and the $\Phi(X_i,\calL)$
    cover $\Phi(X,\calL)$,
  \item the $\Phi(X_i,\calL)$ don't overlap in their interiors, 
  \item the intersection of any two is a face of each, and
  \item the map $Y \mapsto \Phi_T(Y)$ from $T$-orbit closures in $F_0$
    to faces of the subdivision is bijective.
  \end{itemize}

  {\em The covering.}
  The first is easy now, using parts of Proposition \ref{prop:DH}. Each $X_i$ has multiplicity $m_i$, and $d_i$ is the size of  $Stab_T(X_i)/Stab_T(X)$. 
  Each $X_i$ has a dense $T$-orbit by (5), hence contributes $m_i/d_i$ times
  Lebesgue measure on $\Phi(X_i, \calL)$ by a combination of (4), (3), and (1).
  Summing those measures, we get Lebesgue measure on $\Phi(X,\calL)$,
  by a combination of (4) and the other part of (5). That can only happen
  if the $\Phi(X_i,\calL)$ cover $\Phi(X,\calL)$.

  {\em The overlaps.}
  If $\lambda$ is in the interior of both $\Phi(X_i,\calL)$ and $\Phi(X_j,\calL)$,
  we can consider the family
  $F//_\lambda\, T \subseteq (Y//_\lambda\, T) \times \PP^1$ degenerating
  $X//_\lambda\, T$ to $\Union (X_i //_\lambda\, T)$. Since $X$ has a dense
  $T$-orbit and $\lambda \in \Phi(X,\calL)$, the general fiber of this family
  is one point, but the special fiber is at least two points
  (using Proposition \ref{prop:brionprocesi}(2)),
  contradicting Corollary \ref{cor:ZMT} to Zariski's Main Theorem.

  Now that we know the polytopes don't overlap, revisiting the
  sum-of-measures calculation shows that $m_i = d_i$ for each $X_i$.

  {\em The intersections.}
  Let $\lambda \in \Phi(X_i,\calL) \cap \Phi(X_j,\calL)$, not in
  the interior. We want to show that $X_i\cap X_j \neq \emptyset$
  already upstairs in $X$. If $X_i$ and $X_j$ do intersect, then they intersect
  along a $T$-invariant subset, a {\em union} $\Union Y_k$
  of toric subvarieties, which then correspond to full faces of each of
  $\Phi(X_i,\calL), \Phi(X_j,\calL)$.
  Hence $\Phi(X_i,\calL) \cap \Phi(X_j,\calL) \supseteq \Union \Phi(Y_k,\calL)$,
  but as nonoverlapping subpolytopes\footnote{This is a subtle point.
    The two Schubert divisors inside $GL_3/B$ are toric surfaces,
    which intersect along the union of two $\PP^1$s. This is only
    possible because their moment trapezoids have intersecting interiors.}
  of $\Phi(X,\calL)$ their intersection is convex: hence if they
  intersect in a union of faces, they intersect in just one of them.

  Now that we know the proof strategy -- upgrade the polytope intersections
  to subvariety intersections -- we consider the case $\dim T = 1$.
  In this case the polytope statement (about intersections of
  intervals with disjoint interiors) is trivial, so this may seem pointless,
  but we will use the variety statement at $\dim T = 1$ to get the
  general case.

  So, for now $\dim T = 1$. Thus $\Phi(X,\calL)$ is an interval, covered by
  (as just proven) subintervals $\Phi(X_i,\calL)$ with disjoint interiors;
  use this to number the $X_i$ in order (purely for ease of discussion). 
  Since $X$ is connected we know (again using Corollary \ref{cor:ZMT}) that
  $\Union X_i$ is connected.
  Using Proposition \ref{prop:DH}(6) we see $X_i \cap X_j = \emptyset$
  unless $|i-j|=1$, so for $\Union X_i$ to be connected, we need each
  $X_i \cap X_{i+1} \neq \emptyset$. That settles the geometric
  statement for $\dim T=1$, and we move on to general $T$.
  
  By \cite{GR-facet-to-facet},
  it is enough to check the case that
  $\Phi(X_i,\calL) \cap \Phi(X_j,\calL)$ is codimension $1$ in each.
  Pick a rational point $\tilde\lambda$ in the interior of the intersection,
  and a rational line $S^\perp + \tilde\lambda$ transverse to the intersection
  (and the rest of the x-ray), thereby picking a codimension $1$ subtorus
  $S\leq T$.

  Let $\lambda \in \lie s_\RR^*$ be the image of $\tilde\lambda$
  under the projection $\lie t_\RR^* \onto \lie s_\RR^*$.
  Let $E := F//_\lambda\, S$ denote the family degenerating $X//_\lambda\, S$
  to $\Union (X_i//_\lambda\, S)$. The previous analysis applies now
  that $\dim(T/S)=1$, showing that if two $X_i//_S\,\lambda$ have
  abutting moment polytopes, it is because the varieties themselves
  intersect. Hence the $X_i$ intersect (as the choice of $\lambda$
  allows us to use Proposition \ref{prop:brionprocesi}(2)).

  {\em The bijectivity.} We start with the surjectivity. Each face $A$ of the
  subdivision is contained in some facet $A'$, which is the moment
  polytope of some component $C$ of $F_0$. The normalization $\widetilde C$
  of $C$ is the toric variety for $A'$, hence contains a $T$-orbit
  closure $\widetilde Z \subseteq \widetilde C$ with moment polytope $A$.
  If $\nu\colon \widetilde C \to C$ is the normalization map,
  then $\nu(\widetilde Z) \subseteq C \subseteq F_0$ is a $T$-orbit closure
  with moment polytope $A$, as was desired.

  For the injectivity, we want to show that for $T$-orbit closures
  $Y_1 \neq Y_2$, the moment polytopes are different. Since
  $\dim_\FF Y_i = \dim_\RR \Phi_T(Y_i)$ we may assume the orbit
  closures have the same codimension. We prove the (not really) stronger
  statement that if the interiors of $\Phi_T(Y_1)$, $\Phi_T(Y_2)$ meet,
  it is because $Y_1 = Y_2$.
  We have already handled the case of codimension $0$ so we now assume
  the codimensions of each to be $>0$.

  Let $p$ be a rational point in the intersection of the interiors, and
  $L \ni p$ a rational line through $p$ intersecting those polytopes
  only in $p$. (Such an $L$ exists only because both interiors can be
  assumed to have codimension $>0$.) So $L = p + \lie s^\perp$ for a
  codimension $1$ subtorus $S\leq T$. To show $Y_1 = Y_2$, it suffices
  to show the $T/S$-fixed points $Y_1 //_p S, Y_2 //_p S$ are the same.
  The analysis now goes much the same as in the ``The intersections''
  argument above.
\end{proof}

If $F_0$ is reduced as in theorem \ref{thm:Hclass}(1), then it is
generically reduced, so each $m_i=1$ and hence each $d_i=1$.
In particular, the Richardson varieties $X_w^{wc}$ occurring in
theorem \ref{thm:Hclass}(1) have trivial generic $T$-stabilizer
(assuming that the ambient group $G$ was adjoint; otherwise they have
the same generic $T$-stabilizer as $G/B$ has).

Other Richardson varieties can have disconnected generic
$T$-stabilizer; the Richardson variety $X_s^{lsl} \subset SO(5)/B$
(here $s,l$ denote reflections in the short and long simple roots)
has $d=2$, being a closed $SO(4)$-orbit. Is there an embedded degeneration
inside $SO(5)/B$ of a general $T$-orbit closure to the union
$X^s \cup X_s^{lsl} \cup X_{lsl}$, with multiplicity $2$ on the
second component?

\section{Subdivisions of the permutahedron from Coxeter elements}\label{sec:subdiv-permutahedron}

\cref{thm:Hclass} gives an embedded degeneration of the permutahedral variety to a union of Richardson varieties.  \cref{thm:subdivision-general} tells us that this degeneration is in fact semitoric and gives a subdivision of the $W$-permutahedron into the moment polytopes of the Richardson varieties, which are \emph{Bruhat interval polytopes}. In this section, we discuss the combinatorics of this subdivision.

\renewcommand\BB{{(\lie t_\RR^*)_+}}

As in \cref{sec:degen-and-class-formulas}, we fix $G, B, B_-, T$ and
use $W$ to denote the Weyl group. The nodes of the Dynkin diagram of
$W$ are $I$, and $|I|=r$. As in \cref{sec:subdiv-from-degen},
we let $T^*$ be the weight lattice of $T$, with realification $\lie t_\RR^*$.
Assume briefly that $G$ is semisimple simply connected
The weight lattice $T^*$ is spanned by the fundamental weights
$\omega_1, \dots, \omega_r$, which are the dual basis to the simple
coroots. The positive real span of the fundamental weights is the
\defn{fundamental chamber} $\BB \subset \lie t_\RR^*$. The Weyl group
$W$ acts on $\lie t_\RR^*$, and acts freely and transitively on the
complement of the \defn{$W$-braid arrangement}, which consists of the
hyperplanes orthogonal to the coroots of $W$. In type $A$ it will (as always) be more combinatorially convenient to
take $G = GL_n$ than $G = SL_n$, at the cost that the fundamental weights are
not uniquely defined (we {\em choose} $\omega_i$ to be the high weight
$(1,\ldots,1,0,\ldots,0)$ of $\bigwedge^i \CC^n$) and don't span the weight lattice.

\begin{definition}\label{def:BIP-general-type}
	Let $\pt \in \BB$. For $u, v \in W$, the \defn{Bruhat interval polytope} is $\Pi_u^v(\pt):= \conv(z \cdot \pt: z \in [u,v])$. The \defn{$W$-permutahedron} is $\Pi_e^{w_0}(\pt)= \conv(w \cdot \pt: w \in W)$, for which we also use the notation $\Pi(\pt)$. 
\end{definition}

Bruhat interval polytopes for type $A$ were defined in \cite{KW15}, and the definition was expanded to arbitrary type (and to partial flag varieties) in \cite{TW15}\footnote{The work \cite{TW15} takes $\pt=\rho$; since the normal fan does not depend on the choice of $\pt$, this is not a substantial restriction. We also warn the reader that \cite{TW15} is not self-consistent, so their convention on how to label type $A$ Bruhat interval polytopes with Bruhat intervals changes between \cite[Definition 2.2]{TW15} and \cite[Definition 7.8]{TW15}. Our convention matches that of \cite[Definition 7.8]{TW15}.}.

\begin{lemma}\label{lem:BIP-moment-polytope}
	Suppose $\pt \in \BB$ is integral. Then  $\Pi_u^v(\pt)$ is the moment polytope $\Phi(X_u^v, \calL)$ of the Richardson variety $X_u^v$ for some choice of ample $T$-equivariant line bundle $\calL$ on $G/B$. 
\end{lemma}

We note that the normal fan of the Bruhat interval polytope $\Pi_u^v(\pt)$ does not depend on the choice of $\pt$ (see \cite[Lemma 7.2]{BDLSB}). In particular, the normal fan of the $W$-permutahedron is well-known to be the $W$-braid arrangement. The facets of $\Pi(\pt)$ are in bijection with the cosets of maximal parabolic subgroups of $W$.

Because the normal fan does not depend on $\pt$, \cite[Theorem 7.13]{TW15} implies that every face of $\Pi_u^v(\pt)$ is a Bruhat interval polytope $\Pi_p^q(\pt)$ for $[p, q] \subset [u,v]$. When $X_u^v$ is a toric variety, something stronger holds.

\begin{lemma}\label{lem:toric-BIP-facts}
	Suppose $X_u^v$ is a toric variety. Then 
	\begin{enumerate}
		\item for all $[p, q] \subset [u,v]$, $X_p^q$ is toric;
		\item $\Pi_p^q(\pt)$ is a face of $\Pi_u^v(\pt)$ if and only if $[p, q] \subset [u,v]$;
		\item $\dim \Pi_u^v(\pt) = \ell(v)-\ell(u)$;
		\item $\Pi_u^v(\pt)$ does not contain any other Bruhat interval polytopes of the same dimension.
	\end{enumerate}
\end{lemma}
\begin{proof}
  (3) follows from \cref{lem:BIP-moment-polytope}, since toric
  varieties and their moment polytopes have the ``same'' dimension
  (one over $\FF$, one over $\RR$). (4)
  follows from (2), since any Bruhat interval polytope $\Pi_p^q(\pt)$
  contained in $\Pi_u^v(\pt)$ would satisfy $[p, q] \subset [u,v]$,
  and by (2) all such Bruhat interval polytopes are faces of
  $\Pi_u^v(\pt)$ and so have smaller dimension.
	
	We now prove (2). We may assume $\pt$ is integral, and omit it from the notation. Every Richardson variety $X_w^z$ decomposes into open Richardson varieties
	\[X_w^z = \bigsqcup_{[p,q] \subset [w,z]} ({X}_p^q)^\circ\]
	where $({X}_p^q)^\circ = B_-p B/B \cap B q B /B$.
	Each open Richardson variety is $T$-invariant, and so is a union of torus orbits. That is, the decomposition of $X_w^z$ into torus orbits refines the decomposition into open Richardson varieties. Additionally, $X_w^z$ is the closure of the open Richardson variety $({X}_w^z)^\circ$.
	
	Since $X_u^v$ is toric and $\Pi_u^v$ is its moment polytope, the (open) faces of $\Pi_u^v$ are in dimension-preserving bijection with the torus orbits of $X_u^v$. Since the torus orbit decomposition refines the open Richardson decomposition, $X_u^v$ has at least one torus orbit for each subinterval $[p, q] \subset [u,v]$. This means there are at least as many torus orbits as subintervals. On the other hand, each face of $\Pi_u^v$ is of the form $\Pi_p^q$ for some $[p,q] \subset [u,v]$. This means there are at most as many faces as subintervals. Since torus orbits and faces are in bijection, we conclude that there are the same number of faces, torus orbits, and subintervals, so each subinterval gives a face.
	
	For (1), we have just shown that for each $[p,q] \subset [u,v]$, the open Richardson $(X_p^q)^\circ$ is a single torus orbit. Since the closure of $(X_p^q)^\circ$ is $X_p^q$, $X_p^q$ is toric.
\end{proof}

A subdivision of $\Pi(\pt)$ is a \defn{finest Bruhat interval subdivision} if the subdivision consists of Bruhat interval polytopes and no refinement of it consists of Bruhat interval polytopes. The results of \cref{sec:degen-and-class-formulas,sec:subdiv-from-degen} readily give such subdivisions. Recall that $x * y$ denotes the Demazure product of $x$ and $y$.

\begin{theorem}\label{thm:subdivision-Coxeter}
  Let $\pt \in \BB$ and let $c \in W$ be a Coxeter element. For any
  $u$ such that $uc$ is length-additive, $X_u^{uc}$ is toric. The
  collection \begin{equation}\label{eq:BIP-in-subdiv} \{\Pi_u^{uc}(\pt): uc~\text{length-additive}\} \end{equation} is a
  finest Bruhat interval subdivision of the $W$-permutahedron
  $\Pi(\pt)$. The faces of this subdivision are the Bruhat interval
  polytopes $\Pi_y^z$ where $z \leq y * c$.
\end{theorem}

\begin{proof}
  The last sentence of the theorem follows from the prior sentences:
  \cref{lem:toric-BIP-facts} implies that the faces of
  \[ \{\Pi_u^{uc}(\pt): uc~\text{length-additive}\} \]
  are exactly $\Pi_y^z$ where $[y,z] \subset [u,uc]$ for some $uc$
  length-additive. \cref{lem:interval-poset-Demazure-prod} for
  $w = c^{-1}$ shows that this is the same as the intervals $[y,z]$
  such that $z \leq y *c$.
	
	Now we turn to the rest of the theorem statement. We first assume $\pt$ is integral. By \cref{prop:lusztig-permutahedral}, the Lusztig variety $\calY_w$ is a permutahedral variety when $w=c^{-1}$. By \cref{thm:Hclass}, we have an embedded degeneration from $\calY_{c^{-1}}$ to the union
	\[\bigcup\{X_u^{uc}: uc~\text{length-additive}\}.\]
	 \cref{thm:subdivision-general} implies that the Richardsons $X_u^{uc}$ in the above union are toric. The same theorem together with \cref{lem:BIP-moment-polytope} implies that \eqref{eq:BIP-in-subdiv}
	is a subdivision of $\Pi(\pt)$. It is a finest subdivision because, by \cref{lem:toric-BIP-facts}, the $\Pi_u^{uc}(\pt)$ do not contain other full-dimensional Bruhat interval polytopes. Now, \cite[Theorem 7.6]{BDLSB} implies that since \eqref{eq:BIP-in-subdiv} is a subdivision of $\Pi(\pt)$ for one choice of $\pt$, it is a subdivision for all choices of $\pt \in \BB$.
\end{proof}

\section{Regularity in type $A$} \label{sec:regularity-type-A}

In this section, we show in type $A$ that the subdivisions of
\cref{thm:subdivision-Coxeter} are in fact regular subdivisions, by
providing an explicit height vector inducing the subdivisions. This
question is of particular interest because of the relationship between
regular Bruhat interval subdivisions of the type $A$ permutahedron and
the tropical positive flag variety \cite{B23,JLLO23} (see also
\cite{BEW24} and \cref{ssec:rel-to-matroid-polytopes}).

We remind the reader that the results of
\cite{sturmfels1991grobner,zhu2012degenerations} let one use a height
vector to define an embedded (and Gr\"obner) degeneration inside a
large projective space. We do not see how to use the height vector defined
below to produce our tighter (and non-Gr\"obner) embedded
degeneration, inside $G/B$.

We use the notation $[n]:=\{1, \dots, n\}$,
$\binom{[n]}{k}:=\{I \subset [n]: |I|=k\}$ and, for $I \subset [n]$,
$e_I:= \sum_{i \in I} e_i$ is the indicator vector of $I$ in $\RR^n$.
In particular $e_{[i]} = e_1 + \ldots + e_i$.

We work with $G=GL_n$ (rather than $SL_n$ or $PGL_n$, say).
We identify $\lie t_\RR^*$ with $\RR^n$, and choose the fundamental weight $\omega_i$
to be $e_{[i]}$, the highest weight of $\wedge^i \CC^n$. (Because $GL_n$ is not
semisimple, one could make a different choice by tensoring this with a power
of the determinant representation.)
Then $\BB$ is identified with
$\{v=(v_1, \dots, v_n) \in \calH: v_1 > v_2 > \dots > v_n \}$. The
Weyl group $W=S_n$ acts on $\calH$ by permuting coordinates:
$w \cdot (\pt_1, \dots, \pt_n) = (\pt_{w^{-1}(1)}, \dots,
\pt_{w^{-1}(n)})$. We note that $w \cdot e_{[i]} = e_{w[i]}$. For the
remainder of the section, we fix $\pt \in \BB$ and will ease notation
by setting $\Pi:= \Pi_e^{w_0}(\pt)$ and $\Pi_u^v := \Pi_u^v(\pt)$.

\junk{As is standard, we identify $\lie t_\RR^*$ with
  $\RR^n/\RR \cdot \mathbf{1}$. The fundamental weight $\omega_i$ is
  represented by $e_{[i]}$. It is useful to further identify
  $\lie t_\RR^*$ with a hyperplane $\calH$ in $\RR^n$ normal to
  $(1, \dots, 1)$, so $\BB$ is identified with
  $\{v=(v_1, \dots, v_n) \in \calH: v_1 > v_2 > \dots > v_n \}$. The
  Weyl group $W=S_n$ acts on $\calH$ by permuting coordinates:
  $w \cdot (\pt_1, \dots, \pt_n) = (\pt_{w^{-1}(1)}, \dots,
  \pt_{w^{-1}(n)})$. We note that $w \cdot e_{[i]} = e_{w[i]}$. For
  the remainder of the section, we fix $\pt \in \BB$ and will ease
  notation by setting $\Pi:= \Pi_e^{w_0}(\pt)$ and $\Pi_u^v := \Pi_u^v(\pt)$.  }

\begin{remark}
	To obtain the ``usual" permutahedron, that is, the convex hull of all permutation vectors $(v(1), \dots, v(n))$, one should take $\calH$ to be the hyperplane where coordinates sum to $\binom{n+1}{2}$ and take $\pt=(n, n-1, \dots, 1)$. Note that the permutation vector $(v(1), \dots, v(n))$ is \emph{not} equal to $v \cdot (n, n-1, \dots, 1)$, but is instead $ v^{-1} w_0 \cdot (n, \dots, 1)$.
\end{remark}

\subsection{Background}

We recall the basics of regular subdivisions. Here we restrict our attention to regular subdivisions of polytopes where no new vertices are introduced. See \cite{DLRS-triang-book} for additional detail.

\begin{definition}
  Let $P \subset \RR^d$ be a polytope with vertices $V$, and let
  $h:V \to \RR$ be a \defn{height vector}.
  For $v \in V$, $v^h:= (v, h(v))$ is the corresponding \defn{lifted vertex},
  which is a point in $\RR^d \times \RR$. The lifted polytope is
  \[P^h:= \conv\{v^h: v \in V\} \subset \RR^d \times \RR.\]
  A \defn{lower face} of $P^h$ is a face which minimizes a linear functional of the form $\langle (x, 1), - \rangle$.
	The \defn{regular subdivision of $P$ induced by $h$} is the collection of polytopes
	\[\{\conv\{V' \subset V\}: \conv\{v^h: v \in V'\} \text{ is a lower facet of }P^h\}.\]
\end{definition}

\begin{theorem}[{\cite[Theorem 2.3.20]{DLRS-triang-book}}]\label{thm:regular-subdiv-conditions}
	Suppose $\{Q_i\}_{i=1}^m$ is a subdivision of $P$ and $h$ is a height vector on the vertices of $P$. Then $h$ induces the subdivision $\{Q_i\}_{i=1}^m$ if the following two conditions hold.
	\begin{enumerate}
		\item For each $i$, the lifted polytope $Q_i^h$ is contained in a hyperplane $H_i$ (the coplanarity condition);
		\item[(2)] For all facets $F=Q_i \cap Q_j$ of the $\{Q_i\}_{i=1}^m$ which are not in the boundary of $P$, there is a vertex $v \in Q_i \setminus F$ such that $v^h$ lies above $H_j$ (the local folding condition).
	\end{enumerate}
\end{theorem}

Here ``lies above" means that $v^h$ can be obtained from a point in $H_j$ by increasing the last coordinate.

\subsection{The height function and statement of regularity}

We will prove regularity of the subdivision in \cref{thm:subdivision-Coxeter} by exhibiting an explicit height function $\hc$ on the vertices of $\Pi$, and then  use \cref{thm:regular-subdiv-conditions} to verify that $\hc$ induces the desired subdivision.

To define the height function $\hc$ on the vertices $W \cdot \pt$ of $\Pi$, we need the notion of rightmost subexpressions.

\begin{definition}\label{def:rmost-sub}
	Let $\bw=s_{i_1} \dots s_{i_r}$ be a word (not necessarily reduced) in the simple transpositions of $S_n$. A \defn{subexpression} for $v$ in $\bw$ is an expression for $v$ of the form $v=s_{i_1}^v \dots s_{i_r}^v$ where $s_{i_j}^v \in \{e, s_{i_j}\}$. The set of indices $j \in [r]$ where $s_{i_j}^v \neq e$ is the \defn{support} of the subexpression. The \defn{rightmost subexpression}\footnote{Also called the \defn{positive distinguished subexpression} for $v$ in the literature.} for $v$ is constructed using a greedy procedure, moving from right to left, as follows: set  $\vl{r+1}=v$. If $\vl{j+1}$ is already determined, then $\vl{j}$ is equal to either $\vl{j+1}$ or $\vl{j+1}s_{i_{j}}$, whichever is smaller in the Bruhat order. In the first case, $s_{i_j}^v=e$; in the second, $s_{i_j}^v=s_{i_j}$. 
\end{definition}

\begin{remark}
	The rightmost subexpression for $v$ in $\bw$ is indeed the ``rightmost" reduced subexpression for $v$; among all subexpressions with support of size $\ell(v)$, it is the subexpression whose support is colexicographically largest\footnote{A subset $\{a_1< a_2 < \dots < a_r\}$ is larger than $\{b_1 < \dots < b_r\}$ in colexicographic order if there is some $m\in [r]$ such that $a_p= b_p$ for $p >m$ and $a_m>b_m$.}.
\end{remark}

We denote by $S_n\pbolic^{i}:=\{w \in S_n: ws_j>w \text{ for } j \neq i\}$
the minimum length coset representatives in
the maximal quotient $S_n / (S_i \times S_{n-i})$ of $S_n$. For $w \in S_n$, denote by
$w\pbolic^{i} \in S_n\pbolic^{i}$ the minimum length coset
representative in $w(S_i \times S_{n-i})$. That is, $w\pbolic^{i}$ is the shortest
permutation such that $w\pbolic^{i}[i]=w[i]$. In words, $w\pbolic^{i}$
is obtained from $w$ by writing $w_1, \dots, w_i$ in increasing order,
followed by $w_{i+1}, \dots, w_n$ in increasing
order; such permutations are usually called \defn{Grassmannian permutations}.

\begin{definition}\label{def:height-from-c}
	Fix a word $\bc$ for the Coxeter element $c$ and let $\bc^N=s_{i_1} \cdots s_{i_r}$, where $N$ is large enough that $\bc^N$ contains a subexpression for $\wo$. We will weight each letter of $\bc^N$ according to which copy of $\bc$ it is in: for $j \in [r]$, 
	\[\wt(j):=\#\{k\in [j+1, r]: s_{i_k}=s_{i_j}\}.\]
	Given $w \in S_n$, we define 
	$\wt(w):= \sum \wt(j)$,
	where the sum is over the support of the rightmost subexpression for $w$ in $\bc^N$. Finally, we define
	\begin{equation}\label{eq:hc-def}
	\hc(w \cdot \pt):= \sum_{i=1}^{n-1} (\pt_{i} - \pt_{i+1}) \wt(w\pbolic^i).
	\end{equation}
\end{definition}

See \cref{ex:height} for an example.  \cref{def:height-from-c} also has a lattice path interpretation, discussed in \cref{prop:lattice-path-wt-def}. Because the different reduced expressions for $c$ differ only by commutation moves, the height $\hc(w \cdot \pt)$ only depends on $c$, rather than on the choice of reduced word $\bc$.

\begin{example}\label{ex:height}
	Let $\bc=s_1s_2s_3$, and $\bc^N=s_1s_2s_3s_1s_2s_3s_1s_2s_3$. The weights of the letters are $222111000$. Consider $w=3412$. According to the table below, $\hc( w \cdot \pt)= (\pt_1 - \pt_2) \cdot 1 + (\pt_2 - \pt_3) \cdot 2 + (\pt_3 - \pt_4) \cdot 0= \pt_1 + \pt_2 - 2 \pt_3. $
	
	\begin{center}
		\renewcommand{\arraystretch}{1.5}
		\begin{tabular}{|c|c|c|c|c|}
			\hline
			$i$ & $I=w[i]$ & $w\pbolic^{i}$ & rightmost subexp. for $w\pbolic^{i}$ in $\bc^N$ & $\wt(w\pbolic^{i})$ \\
			\hline
			$1$ & \{3\} & 3124 & $s_1s_2s_3s_1\boxed{s_2}s_3\boxed{s_1}s_2s_3$ & $1+0=1$ \\
			\hline
			$2$ & \{3,4\} & 3412 & $s_1s_2s_3s_1\boxed{s_2}\boxed{s_3}\boxed{s_1}\boxed{s_2}s_3$ & $1+1+0+0=2$ \\
			\hline
			$3$ & \{1,3,4\}   & 1342 & $s_1s_2s_3s_1s_2s_3s_1\boxed{s_2}\boxed{s_3}$ & $0+0=0$\\
			\hline
		\end{tabular}
	\end{center}
	
\end{example}

The main result of this section is the following.

\begin{theorem}\label{thm:regular-type-A}
  Let $\pt \in \BB$, let $c \in S_n$ be a Coxeter element, and let
  $\hc$ be the height function from \cref{def:height-from-c}. The collection
  \[ \{\Pi_u^{uc}(\pt): uc~\text{length-additive}\} \]
  is a finest regular Bruhat interval subdivision of the permutahedron
  $\Pi(\pt)$, induced by $\hc$.
\end{theorem}

\begin{remark} 
	Recall that the \defn{hypersimplex} $\Delta_{i,n}$ is the polytope $\conv(W \cdot e_{[i]}) = \conv\{e_I : I \in \binom{[n]}{i}\}$. The permutahedron $\Pi$ is a Minkowski sum of dilated hypersimplices
	\[\Pi(\pt) = (\pt_1 - \pt_2) \Delta_{1,n}+ (\pt_2 - \pt_3) \Delta_{2,n}+ \dots +(\pt_{n-1} - \pt_n) \Delta_{n-1, n} + \pt_n \Delta_{n,n} \]
	and in particular each vertex decomposes uniquely as a sum of hypersimplex vertices
	\[w \cdot \pt =(\pt_1 - \pt_2) e_{w[1]}+ (\pt_2 - \pt_3) e_{w[2]}+ \dots +(\pt_{n-1} - \pt_n) e_{w[n-1]}+ \pt_n e_{w[n]}. \]
One can interpret \cref{def:height-from-c} as defining a height function on vertices of hypersimplices $\Delta_{i,n}$, where the height of $e_I=e_{w[i]}$ is $\wt(w\pbolic^{i})$. (And the height function on $\Delta_{n,n} = e_{[n]}$ is zero.) Then $\hc$ is the height function on the vertices of $\Pi$ induced by writing each vertex as a sum of hypersimplex vertices and then taking the sum of heights.
\end{remark}

\begin{remark}
  \cref{def:height-from-c} has an obvious analog for arbitrary Weyl
  groups. However, the analogous height function $\hc$ does not induce
  the subdivision
  \[ \{\Pi_u^{uc}(\pt): uc~\text{length-additive}\} \]
  in other types. For example, in type $D$, to obtain this subdivision
  one must adjust the height function on the type D hypersimplices
  $\conv(W \cdot \omega_i)$ corresponding to non-minuscule nodes of
  the Dynkin diagram.
\end{remark}

\subsection{Coplanarity}

As in the previous section, we have fixed $\pt \in \BB$. We now fix a Coxeter element $c \in S_n$ and in so doing fix the height vector $\hc$ as defined in \cref{def:height-from-c}. In this section, we show that each Bruhat interval polytope in $\{\Pi_{u}^{uc}: uc \text{ length-additive}\}$ lifts to a plane, verifying the coplanarity condition of \cref{thm:regular-subdiv-conditions}.

To ease notation, we write $\hc(v)$ for $\hc(v \cdot \pt)$.  If $\lin: \RR^{n+1} \to \RR$ is a linear functional, we write $\lin_{a}:=\lin(e_a)$ for the image of the standard basis vector $e_a$. The following proposition is the main result of this subsection.

\begin{proposition}\label{prop:good-bip-lifted-to-plane}
	Suppose $w\in S_n$ and $wc$ is length-additive. Then there is a linear functional $\lin: \RR^{n+1} \to \RR$ which is constant on $([w,wc] \cdot \pt)^{\hc}$ and has $\lin_{n+1}=1$.
\end{proposition}

We will need the characterization of length 2 intervals in Bruhat order \cite{BGG73}.

\begin{lemma}\label{lem:bruhat-length-2-facts}
  Suppose $u \leq v$ and $\ell(v)-\ell(u) =2$. Then $[u,v]$ is a
  diamond and there is a transposition $t$ such that the open Bruhat
  interval $(u,v)$ is $\{tu, tv\}$. The Hasse diagram of the closed
  interval is as below.
		\[\begin{tikzcd}
		& {v} \\
		{tu} && {tv} \\
		& u
		\arrow[no head, from=1-2, to=2-3]
		\arrow[no head, from=2-1, to=1-2]
		\arrow[no head, from=2-3, to=3-2]
		\arrow[no head, from=3-2, to=2-1]
	\end{tikzcd}\]
\end{lemma}

For $\hgt:S_n \cdot \pt \to \RR$ a height function, $v \in S_n$ and $t=(a~b)$ a transposition with $a<b$, we use the notation
\begin{equation}\label{eq:incline}
	\steep(v, t):=\frac{\hgt(t v)-\hgt(v)}{\pt_{v^{-1}(a)}-\pt_{v^{-1}(b)}}.
\end{equation}
This quantity measures the ``incline" of the line segment between $(v \cdot \pt)^\hgt$ and $(tv \cdot \pt)^\hgt$, which is parallel to 
\[e_a -e_b + \frac{\hgt(t v)-\hgt(v)}{\pt_{v^{-1}(a)}-\pt_{v^{-1}(b)}}  e_{n+1}.\]

Using inclines, we formulate a necessary and sufficient condition for \cref{prop:good-bip-lifted-to-plane} to hold.

\begin{proposition}\label{prop:ratios-for-coplanarity}
	Fix a height function $\hgt:S_n \to \RR$.
	\begin{itemize}
		\item 	 Let $v \in S_n$ and $t=(a~b)$ be a transposition with $a<b$. A linear functional $\lin: \RR^n \times \RR \to \RR$ with $\lin_{n+1} \neq 0$ takes the same value on $(v \cdot \pt)^\hgt$ and $(tv \cdot \pt)^\hgt$ if and only if
		\[\steep(v,t)= \frac{\lin_a-\lin_b}{\lin_{n+1}}.\]
		\item 	If $X_p^q$ is toric, then there exists a linear functional $\lambda: \RR^n \times \RR \to \RR$ with $\lin_{n+1} =1$ which is constant on $([p,q] \cdot \pt)^{\hgt}$ if and only if for all subintervals of rank 2
				\[\begin{tikzcd}
			& {tv} \\
			{tu} && v \\
			& u
			\arrow[no head, from=1-2, to=2-3]
			\arrow[no head, from=2-1, to=1-2]
			\arrow[no head, from=2-3, to=3-2]
			\arrow[no head, from=3-2, to=2-1]
		\end{tikzcd}\]
	 we have
		\begin{equation}\label{eq:ratios-for-lifts-coplanar}
			\steep(v,t) = \steep(u,t).
		\end{equation}
	\end{itemize}
\end{proposition}

\begin{remark}
	For $X_p^q$ toric, rank 2 sub-intervals of $[p,q]$ correspond exactly to dimension 2 faces of $\Pi_p^q$. The edge between vertices 
	$v \cdot \pt$ and $tv \cdot \pt$ is parallel to the root $\alpha_t$, as is the edge between $u \cdot \pt$ and $tu \cdot \pt$. When we lift these vertices, these edges lift to line segments parallel to 
	\[\alpha_t + \steep(v,t)  e_{n+1} \qquad \text{and}\qquad \alpha_t + \steep(u,t)  e_{n+1},\]
	respectively. 
	So Condition \eqref{eq:ratios-for-lifts-coplanar} in \cref{prop:ratios-for-coplanarity} is the same as saying ``for each quadrilateral face of $\Pi_p^q$, the parallel edges lift to parallel line segments."
\end{remark}

\begin{proof}[Proof of \cref{prop:ratios-for-coplanarity}]
	For the first item, recall that $v \cdot \pt = (\pt_{v^{-1}(1)}, \dots, \pt_{v^{-1}(n)})$ and $tv \cdot \pt$ is obtained from $v \cdot \pt$ by swapping the $a$th and $b$th coordinates. Consider the ray from the lifted point $(v \cdot \pt)^\hgt$ to $(t v \cdot \pt)^\hgt$. We have
	\[\langle \lin, (\pt_{v^{-1}(a)}-\pt_{v^{-1}(b)}) (e_a-e_b) + (\hgt(v) - \hgt({tv}))e_{n+1} \rangle= (\pt_{v^{-1}(a)}-\pt_{v^{-1}(b)}) (\lin_a - \lin_b)+(\hgt(v) - \hgt({tv})) \lin_{n+1}. \]
	The above equation is equal to zero if and only if $\steep(v,t)=({\lin_a-\lin_b})/{\lin_{n+1}}$.
	
	For the second item, if such a linear functional $\lin$ exists, the first item implies
	\[\steep(v,t)=\frac{\lin_a -\lin_b}{\lin_{n+1}}=\steep(u,t) \]
any time $v, tv, u, tv$ are all in the interval $[p,q]$. So \eqref{eq:ratios-for-lifts-coplanar} holds. So we just need to show that \eqref{eq:ratios-for-lifts-coplanar} implies the existence of $\lin$.
	
	Fix a maximal chain $\mcc=p=p_1 \lessdot p_2 \lessdot \cdots \lessdot p_{r+1} =q$, and let $t_i=(a_i ~ b_i) :=p_{i+1} p_i^{-1}$. In light of the first item, we would like to choose a vector $\lin \in (\RR^{n+1})^*$ with $\lin_{n+1} =1$ such that for all $i=1, \dots r$, 
	\begin{equation}\label{eq:ratio-along-chain}
			\lin_{a_i}-\lin_{b_i}
		= \steep(p_i, t_i).
	\end{equation}
	This is indeed possible. By \cite[Proposition 4.12]{TW15}, since $X_p^q$ is toric, the graph $G_\mcc$ on $[n]$ with edges $\{\{a_i, b_i\}: i=1, \dots, r\}$ is a forest. We will define $\lin$ one connected component of $G_\mcc$ at a time. First, set $\lin_{n+1}=1$. Next, choose a connected component of $G_\mcc$, choose a root $k$, and choose any value for $\lin_k$. Now, for any non-root vertex $k'$ of this component, the condition \eqref{eq:ratio-along-chain} uniquely determines the value of $\lin_{k'}$. One can see this by traversing the vertices in e.g. depth-first-search order; when one reaches $k'$, the value $\lin_m$ for the parent $m$ of $k'$ has already been determined, so \eqref{eq:ratio-along-chain} can be solved for $\lin_{k'}$. 
	
	By construction, $\lin$ is constant on the lift of $\mcc \cdot \pt$. 
	We now will show that in fact $\lin$ is constant on the lift of every maximal chain in $[p,q]$, and thus on $([p,q] \cdot \pt)^\hgt$.
	
	First, if $\ell(q)-\ell(p)=1$, there is nothing to prove. If $\ell(q)-\ell(p)=2$, then by \cref{lem:bruhat-length-2-facts}, we have 
			\[\begin{tikzcd}
		& {q}=tx \\
		{tp} && {tq}=x \\
		& p
		\arrow[no head, from=1-2, to=2-3]
		\arrow[no head, from=2-1, to=1-2]
		\arrow[no head, from=2-3, to=3-2]
		\arrow[no head, from=3-2, to=2-1]
	\end{tikzcd}\]
	Suppose $\mcc$ contains $tp$ (the other case is identical). We have
		\[\frac{\lin_a - \lin_b}{\lin_{n+1}} = \steep(p, t)=\steep(x,t)\]
		where the first equality follows from \eqref{eq:ratio-along-chain} and the second is assumption \eqref{eq:ratios-for-lifts-coplanar}. By the first item, this implies that $\lin$ takes the same value on $(x \cdot \pt)^\hgt$ as on $(q \cdot \pt)^\hgt$ and so takes the same value on the entire lifted interval, as desired.
	
	If $\ell(q)-\ell(p)>2$, the previous paragraph implies that if $\lin$ is constant on $(\mcc' \cdot \pt)^\hgt$ for a maximal chain $\mcc' \subset [p,q]$, and $\mcc'' \subset [p,q]$ is a maximal chain which differs from $\mcc'$ in precisely one element, then $\lin$ is also constant on $(\mcc'' \cdot \pt)^\hgt$. The order complex of $(p,q)$ is a PL-sphere of dimension at least 1 \cite{BW82}, and is thus connected in codimension 1. This means there is a sequence $\mcc=\mcc_1, \dots, \mcc_N$ containing all maximal chains of $[p,q]$ such that adjacent chains in the sequence differ by exactly one element. All elements of $[p,q]$ are in some maximal chain, so the result follows.
\end{proof}

In order to use \cref{prop:ratios-for-coplanarity} for $\Pi_w^{wc}$, we will translate the definition of the height function $\hc$ into lattice paths. We first need some notation relating lattice paths and reduced expressions.

Let $\grid_n$ denote the collection of boxes in the integer grid which lie between the lines $y=-x$ and $y=-x-n$. The points of $\grid_{n}$ that lie on the line $y=-x$ are the \defn{upper right corners} of $\grid_n$, and those that lie on $y=-x-n$ are the \defn{lower left corners}. The \defn{content} of a box in $\grid_n$ is the positive integer $i$ such that the line $y=-x-i$ passes through the interior of the box.
Content increases as one travels down or to the left in $\grid_n$, and is constant as one travels parallel to $y=-x$.

\begin{figure}
	\centering
	\includegraphics[width=0.5\linewidth]{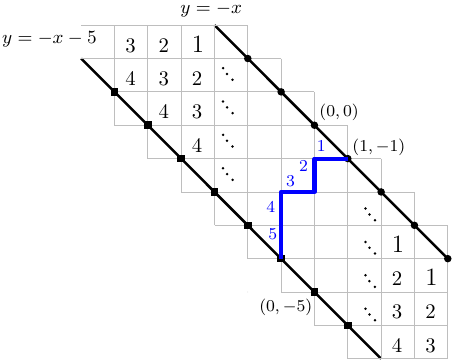}
	\caption{$\grid_5$, with upper right corners marked with dots, and lower left corners marked with diamonds. The numbers inside boxes indicate their content. In bold, the path $\path_{I}^{(1,-1)}$ for $I=\{2,4,5\}$.}
	\label{fig:grid-with-content}
\end{figure}

A \defn{lattice path} in $\grid_n$ is a path which starts at an upper right corner of $\grid_n$, takes unit length steps down or to the left, and ends at a lower left corner of $\grid_n$. We label the steps of the path with $1, 2, \dots, n$ so labels increase from beginning to end. If $I \subset [n]$ and $r$ is an upper right corner of $\grid_n$, we write $\path_I^r$ for the lattice path starting at $r$ with vertical steps $I$. We write $\path_I$ for $\path_I^{(0,0)}$. We note that if $i \in I$, step $i$ of $\path_I$ is immediately to the right of a box with content $i$; if $i \notin I$, step $i$ is immediately below a box with content $i-1$.

Recall that if $I=\{i_1 < i_2< \cdots < i_k\} \subset [n]$ and $J=\{j_1 < j_2< \cdots < j_k\} \subset [n]$, then $I \leq J$ in the \defn{Gale order} if $i_p \leq j_p$ for $p =1, \dots, k$.

\begin{definition}\label{def:skew-shape-btw-paths}
	If $I,J \subset[n]$ have size $k$ and $I\leq J$ in the Gale order, the skew shape $\skew_{I/J}^r$ consists of the boxes in $\grid_n$ that lie above $\path_{I}^r$ and below $\path_{J}^r$. 
	We also write $\skew_{J}^r$ for $\skew_{[k]/J}^r$.  
	We call $\path_I^r$ the \defn{bottom boundary} of $\skew_{I/J}^r$, and $\path_J^r$ the \defn{top boundary}.
	
	Let $c \in S_n$ be a Coxeter element. Let $I:= \{i\in [n]: i<c(i)\}$ be the excedance set of $c$ and define $J = (I \setminus \{1\}) \cup \{n\}$. The \defn{ribbon strip} for $c$ is $\rib_c:=\skew_{I/J}^{(0,0)}$.

\end{definition}

See \cref{fig:ribbon-ex} for an example of $\rib_c$.

\begin{figure}
	\centering
	\includegraphics[width=0.9\linewidth]{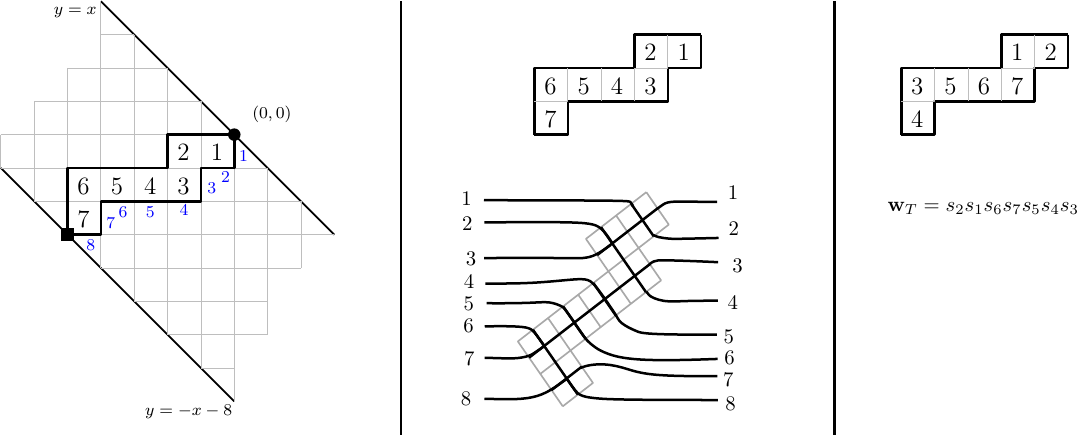}
	\caption{Left: The ribbon strip $\rib_c$ for $c=s_6 s_5 s_4 s_2 s_7 s_3 s_1= 31724586$ in $\grid_8$. The excedance set of $c$ is $I=\{1,3,7\}$, and $\path_I$ is the lower boundary of $\rib_c$. Boxes of $\rib_c$ are filled with their content. Center: The wiring diagram for $c$ obtained by placing a ``cross" in each box of $\rib_c$ and rotating $45^\circ$ counterclockwise. Right: A standard Young tableau $T$ of shape $\rib_c$, and the corresponding reduced expression $\bw_T$ for $c$ under the bijection of \cref{lem:SYT-red-expr-bij}. The content of the box containing $1$ gives the first letter of $\bw_T$; the content of the box containing $n-1=7$ gives the last letter.}
	\label{fig:ribbon-ex}
\end{figure}

Note that for $r$ fixed, the map $w \mapsto \path_{w[k]}^r$ gives a bijection between $S_n\pbolic^{k}$ and lattice paths in the rectangle $\skew_{[n-k+1, n]}^r$.

\begin{lemma}\label{lem:SYT-red-expr-bij}
	Let $w \in S_n\pbolic^{k}$ (resp, $w=c$). For a standard Young tableau of shape $\skew_{w[k]}^r$ (resp. $\rib_c$), let $j_i$ denote the content of the box where $i$ appears in $T$. Then the map 
	\[T \mapsto \bw_T:=s_{j_1} s_{j_2} \dots s_{j_\ell} \]
	is a bijection between standard Young tableaux $T$ of shape $\skew_{w[k]}^r$ (resp. $\rib_c$) and reduced expressions $\bw$ for $w$. 
\end{lemma}

See \cref{fig:ribbon-ex} for an illustration of \cref{lem:SYT-red-expr-bij}.

\begin{proof}
	The permutation $w$ is \defn{fully commutative}, i.e. all reduced expressions are related by commutation moves. That is, $w$ has a single wiring diagram, and all reduced expressions for $w$ are obtained by numbering the crossings of this wiring diagram so that traveling along a wire from left to right, the wire passes through an increasing sequence of crossings, then writing down the simple transpositions corresponding to the crossings in order from smallest numbered crossing to largest. 
	
	If we place a ``cross" in each box of the skew shape (and tilt the shape $45^\circ$ counterclockwise), we obtain a wiring diagram, with endpoints of wires labeled $1, \dots, n$ reading top to bottom (see \cref{fig:ribbon-ex}, center). The content of a box gives the simple transposition corresponding to the crossing in that box. Wires go either straight up or straight to the left, so no two wires cross more than once. The wire which has right endpoint on the step labeled $i$ has left endpoint on the step labeled $w(i)$, so this is a wiring diagram for $w$. 
	Traveling along a wire from left to right, we see that the wires in box $b$ go into the box immediately to the right and the box immediately below $b$. So the condition above that wires travel through an increasing sequence of crosses is the same as the condition that box $b$ has a lower number than the box to the right and the box below. Thus, the numberings mentioned in the paragraph above are in bijection with standard Young tableaux of shape $\skew_{w[k]}^r$ (resp. $\rib_c$).
\end{proof}

For a fixed Coxeter element $c \in S_n$, we define a sequence of upper right corners $r_1, \dots, r_{n-1}$ of $\grid_{n}$ by 
\begin{equation}\label{eq:corner-seq}
	r_1:=(0,0) \quad \text{and} \quad r_i:=\begin{cases}
		r_{i-1} & \text{if }c(i)>i\\
		r_{i-1} + (-1,1) & \text{if } c(i)<i.
	\end{cases}
\end{equation}
These corners are defined so that for any $I \subset[n]$ of cardinality $k$, the box in the lower right corner of $\skew^{r_k}_{I}$ is the unique box of $\rib_c$ with content $k$. See \cref{fig:weight-ex}, left, for an example.

We weight the boxes of $\grid_{n}$ above $\rib_c$ according to which translate of $\rib_{c}$ they lie in. 

\begin{definition}\label{def:weight}
	A box $b \in \grid_n$ which lies in $\rib_c^{(-i,i)}$ has \defn{weight} $\wt(b):=i$. The weight of a collection $C$ of boxes is $\wt(C):= \sum_{b \in C} \wt(b)$.
\end{definition}
See \cref{fig:weight-ex}, center and right, for an example.

\begin{figure}
	\centering
	\includegraphics[width=0.7\linewidth]{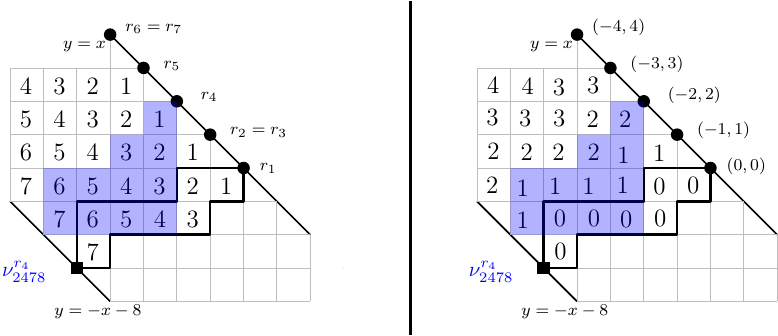}
	\caption{Left: The sequence of upper right corners $r_1, \dots, r_7$ for $c=31724586$. In black, the ribbon $\rib_c$. Shaded in blue, the skew shape $\skew^{r_4}_{2478}$. Boxes are filled with their content. The lower right corner box of $\skew^{r_4}_{2478}$ is the unique box of $\rib_c$ of content 4. By \cref{lem:SYT-red-expr-bij}, reading the content of boxes of $\skew^{r_4}_{2478}$ in the order given by a standard Young tableau gives a reduced expression for $v=24781256 \in S_8\pbolic^{4}$. Right: The boxes of $\grid_{n}$ filled with their weight. The weight of $\skew^{r_4}_{2478}$ is $10$, the sum of weights of boxes in $\skew^{r_4}_{2478}$.}
	\label{fig:weight-ex}
\end{figure}

We may rephrase the height function $\hc$ using weights of certain skew shapes.

\begin{proposition}\label{prop:lattice-path-wt-def}
	Let $w \in S_n$. Then 
	\[\hc(w)= \sum_{i=1}^{n-1} (\pt_i - \pt_{i+1}) \wt \left(\skew_{w[i]}^{r_i}\right).\] 
\end{proposition}
\begin{proof}
	Recall from \cref{def:height-from-c} the word-theoretic definition of $\hc$. We must compare the weight of the rightmost subexpression for $w\pbolic^{i}=:u$ in $\bc^N$ to the weight of the skew shape $\skew_{w[i]}^{r_i}$. In the first case, the weight is determined entirely by how many letters of the rightmost subexpression lie in each copy of $\bc$. In the second case, the weight is determined entirely by how many boxes of $\skew_{w[i]}^{r_i}$ lie in each translate of $\rib_c$. 
	
	Consider the region $\rib_c \cup \rib_c^{(-1,1)} \cup \cdots \cup \rib_c^{(-N+1, N-1)}$ in $\grid_n$. Similarly to the proof of \cref{lem:SYT-red-expr-bij}, if we place a ``cross" in each box of this region, we obtain a wiring diagram for $\bc^N$. Regardless of the choice of $\bc$, the crosses corresponding to the rightmost expression for $u$ lie exactly in the boxes of $\skew_{w[i]}^{r_i}$. Indeed, $u \in S_n\pbolic^{i}$, so its unique right descent is $s_i$. The last letter of the rightmost subexpression for $u$ is thus the rightmost $s_i$ in $\bc^N$. The corresponding cross lies in the unique box of $\rib_c$ with content $i$; by the choice of $r_i$, this is exactly the lower right corner of $\skew_{u[i]}^{r_i}$. By \cref{lem:SYT-red-expr-bij}, all reduced expressions for $u$ correspond to standard Young tableaux of shape $\skew_{u[i]}^{r_i}$ and in particular the crosses used lie exactly in the skew shape $\skew_{u[i]}^{r_i}=\skew_{w[i]}^{r_i}$. By the definition of weight, a cross lies in a box with weight $j$ if and only if the corresponding simple transposition in $\bc^N$ has weight $j$. This completes the proof.
\end{proof}

Now that we have an interpretation of $\hc$ in terms of skew shapes, the next step is to interpret \eqref{eq:ratios-for-lifts-coplanar} in terms of skew shapes. We start with a helpful lemma.

\begin{lemma}\label{lem:w-wc-bound-ribbon}
	Suppose $wc$ is length-additive. Fix $i \in [n-1]$. Let $j$ (resp. $k$) be the step of the top boundary of $\rib_c$ which lies above (resp. to the left) of the box in $\rib_c$ with content $i$. The skew shape $\skew_{w[i]/wc[i]}^{r_i}$ is a ribbon shape with content $\{w_j, w_j +1 , \dots, w_k-1\}$.
\end{lemma}
\begin{proof}
	The content does not depend on the corner $r_i$, so we will work with $\skew_{}:=\skew_{w[i]/wc[i]}$. 
	Recall from the proof of \cref{lem:SYT-red-expr-bij} that $\rib_c$ gives a wiring diagram for $c$. So if one considers step $a$ of the bottom boundary of $\rib_c$, the unique step of the top boundary which is in either the same row or same column is step $c_a$. Using this, one can determine that $c[i]=[i] \setminus \{j\} \cup \{k\}$. So $wc[i]=w[i] \setminus \{w_j\} \cup \{w_k\}$. Since $wc>w$, we must have $w_j < w_k$. This means exactly that $\skew_{}$, whose top boundary is $\path_{wc[i]}$ and whose bottom boundary is $\path_{w[i]}$, is a ribbon shape with content $\{w_j, w_j +1, \dots, w_k-1\}$.  
\end{proof}

\begin{lemma}\label{prop:wt-constant}
	Suppose $wc$ is length-additive. Suppose $w \leq u <  (a~b) u=:v \leq wc$ with $a<b$. Let $\sigma_i \subset \skew_{w[i]/wc[i]}^{r_i}$ be the sub-ribbon strip with content $a, a+1, \dots, b-1$ (note that this does not depend on $u$). Then
	\begin{itemize}
			\item $\skew_{u[i]/v[i]}^{r_i}$ is empty for $ i \notin [u_a^{-1}, u_b^{-1})$ and $\skew_{u[i]/v[i]}^{r_i}=\sigma_i$ for $i \in [u_a^{-1}, u_b^{-1})$.
		\item for all $i, j \in [u_a^{-1}, u_b^{-1})$, $\sigma_i = \sigma_j$.
	\end{itemize}
	In particular, $\wt(\skew_{u[i]/v[i]}^{r_i})$ is the same for all $i \in [u_a^{-1}, u_b^{-1})$.
\end{lemma}

\begin{proof}
	If $i \notin [u^{-1}_a, u^{-1}_b)$, then $u[i]=v[i]$. The skew shape $\skew_{u[i]/v[i]}^{r_i}$ is empty as claimed.
	
	If $i \in [u^{-1}_a, u^{-1}_b)$, then $v[i]=u[i] \setminus \{a\} \cup \{b\}$. This means that $\skew_{u[i]/v[i]}^{r_i}$ is a ribbon strip of content $a, a+1, \dots, b-1$. On the other hand, $wc[i] \geq v[i] \geq u[i] \geq w[i]$ in the Gale order, so $\skew_{u[i]/v[i]}^{r_i}$ is contained in $\skew_{w[i]/wc[i]}^{r_i}$, which is a ribbon strip by \cref{lem:w-wc-bound-ribbon}. Thus, $\skew_{u[i]/v[i]}^{r_i}= \sigma_i$. 
	
	We next show that if $i, i+1\in [u^{-1}_a, u^{-1}_b)$, then $\sigma_i=\sigma_{i+1}$; this immediately implies $\sigma_i=\sigma_{j}$ for all $i, j  \in [u^{-1}_a, u^{-1}_b)$. In other words, we will show $\skew_{w[i]/wc[i]}^{r_i}$ and $\skew_{w[i+1]/wc[i+1]}^{r_{i+1}}$ contain the same boxes of content $a, \dots, b-1$. There are two cases. 
	
	\noindent \textbf{Case 1:} Suppose $r_{i+1}= r_i$, which by \eqref{eq:corner-seq} is equivalent to $c_{i+1}>i+1$. In this case, step $i+1$ of the bottom boundary of $\rib_c$ is vertical. This means that in $\rib_c$, the box of content $i+1$ has a box $b$ of content $i$ directly above it.  
	The step on the top boundary to the left of $b$ is step $i+1$; the step above $b$ is $j$ for some $j$. So by \cref{lem:w-wc-bound-ribbon}, $\skew_{w[i]/wc[i]}^{r_i}$ is a ribbon strip with content $\{w_j, w_j+1, \dots, w_{i+1}-1\}$.
	In particular, since $\skew_{w[i]/wc[i]}^{r_i}$ contains $\sigma_i$ and thus contains a box with content $b-1$, we have that $w_{i+1} \geq b$. Now, we compare the bottom boundaries of $\skew_{w[i]/wc[i]}^{r_i}$ and $\skew_{w[i+1]/wc[i+1]}^{r_{i+1}}=\skew_{w[i+1]/wc[i+1]}^{r_{i}}$. Since $w[i+1]=w[i] \cup \{w_{i+1}\}$, the paths $\path_{w[i+1]}^{r_i}$ and $\path_{w[i]}^{r_i}$ agree through step $w_{i+1}-1$. The box of content $k$ in a ribbon shape $\rib_{}$ is exactly the box of content $k$ in $\grid_n$ whose bottom right corner touches the bottom boundary of $\rib_{}$. So, if both $\skew_{w[i]/wc[i]}^{r_i}$ and $\skew_{w[i+1]/wc[i+1]}^{r_{i+1}}$ contain some box of content $k \leq w_{i+1}$, they in fact contain the same box of content $k$. Since $w_{i+1}\ge b$, this shows that $\sigma_i=\sigma_{i+1}$ as desired.
	
	\noindent \textbf{Case 2:} Suppose $r_{i+1}\neq r_i$. The argument here is similar. In this case, step $i+1$ of the bottom boundary of $\rib_c$ is horizontal and is the bottom edge of the box of content $i$. This means that in $\rib_c$, the box of content $i$ has a box $b$ of content $i+1$ directly to its left. The step of the top boundary above $b$ is step $i+1$, and the step to the left of $b$ is step $k$ for some $k$. Again appealing to \cref{lem:w-wc-bound-ribbon}, $\skew_{w[i+1]/wc[i+1]}^{r_{i+1}}$ has content $\{w_{i+1},w_{i+1}+1,\dots, w_k-1\}$. Since $\skew_{w[i+1]/wc[i+1]}^{r_{i+1}}$ contains $\sigma_{i+1}$ and thus a box with content $a$, we have $w_{i+1} \leq a$. Now, we again compare the bottom boundaries of $\skew_{w[i]/wc[i]}^{r_i}$ and $\skew_{w[i+1]/wc[i+1]}^{r_{i+1}}$. Recalling that $r_{i+1}=r_i + (-1,1)$ by \eqref{eq:corner-seq}, we see that $\path_{w[i]}^{r_i}$ and $\path_{w[i+1]}^{r_{i+1}}$ agree on steps $w_{i+1}+1, \dots, n$. So, if $\skew_{w[i]/wc[i]}^{r_i}$ and $\skew_{w[i+1]/wc[i+1]}^{r_{i+1}}$ each contain some box of content $k \ge w_{i+1}$, they in fact contain the same box of content $k$. Since $a \ge w_{i+1}$, this shows that $\sigma_i=\sigma_{i+1}$ as desired.    
\end{proof}

Now, we use \cref{prop:wt-constant} and \cref{prop:ratios-for-coplanarity}
to prove \cref{prop:good-bip-lifted-to-plane}.

\begin{proof}[Proof of \cref{prop:good-bip-lifted-to-plane}]
	Fix $wc$ length-additive, so $X_w^{wc}$ is toric. We will show that \eqref{eq:ratios-for-lifts-coplanar} holds for the height function $\hgt=\hc$. 
	
	Let $u < t u=x$ be elements of $[w,wc]$, where $t = (a~b)$ with $a<b$. As a first step, we give a skew-shape interpretation for 
	\[\steep(u, t) = \frac{\hc(x)-\hc(u)}{\pt_{u^{-1}(a)}-\pt_{u^{-1}(b)}}.\]
	
	As in \cref{prop:wt-constant}, define $\sigma_i\subset \skew_{w[i]/wc[i]}^{r_i}$ to be the sub-ribbon with content $\{a, a+1, \dots, b-1\}$; note that this does not depend on $u$.  Fix $j$ such that $u[j] \neq x[j]$. We have
	\begin{align*}\hc(x)-\hc(u )&=\sum_{i=1}^{n-1} (\pt_i - \pt_{i+1}) \wt(\skew_{x[i]}^{r_i})-\wt(\skew_{u[i]}^{r_i})= \sum_{i=1}^{n-1}(\pt_i - \pt_{i+1}) \wt(\skew_{u[i]/x[i]}^{r_i})\\
		&= \sum_{\substack{i \in [n-1]:\\u[i] \neq x[i]}} (\pt_i - \pt_{i+1}) \wt(\sigma_i) = \wt(\sigma_j) \sum_{i \in [u^{-1}_a, u^{-1}_b)}\pt_i - \pt_{i+1} = \wt(\sigma_j) (\pt_{u^{-1}(a)}-\pt_{u^{-1}(b)})\end{align*}
	where the first equality follows from \cref{prop:lattice-path-wt-def}, the second follows from the definition of weight, and the third is \cref{prop:wt-constant}. This implies $\steep(u,t)=\wt(\sigma_j)$. 
	
	Now, consider a rank 2 subinterval 
\[\begin{tikzcd}
	& tv \\
	tu && v \\
	& u
	\arrow["{t_2}"', no head, from=2-1, to=1-2]
	\arrow["t", no head, from=2-3, to=1-2]
	\arrow["t"', no head, from=3-2, to=2-1]
	\arrow["{t_1}", no head, from=3-2, to=2-3]
\end{tikzcd}\]
and suppose $v=t_1 u$ and $tv=t_2tu$ for transpositions $t_1, t_2$. We have that $\steep(u,t)=\wt(\sigma_j)$, for any $j$ such that $u[j] \neq tu[j]$, and $\steep(v,t)=\wt(\sigma_{j'})$, for any $j'$ such that $v[j'] \neq tv[j']$. We would like to show that $\steep(u,t)=\steep(v,t)$. It suffices to show that we may choose $j=j'$, that is, that $[u_a^{-1}, u_b^{-1}) \cap [v_a^{-1}, v_b^{-1})$ is nonempty. If $t_1=t_2$, then $t$ and $t_1$ commute, so $[u_a^{-1}, u_b^{-1}) = [v_a^{-1}, v_b^{-1})$. Otherwise, we have that $\{t_1, t_2\} =\{ (a~x),(b~x)\}$ for some $x$. Casework shows that $[u_a^{-1}, u_b^{-1}) \cap [v_a^{-1}, v_b^{-1})$ is nonempty.

\end{proof}

\subsection{Local folding and proof of \cref{thm:regular-type-A}}

Throughout this subsection, we set $\hgt=\hc$ and for $v \in S_n$, we use the notation $v^{\hgt}$ for $(v \cdot \pt, \hc(v \cdot \pt))$. For $w \in S_n$ such that $wc$ is length-additive, we use the notation $\lin^w$ for the linear functional from \cref{prop:good-bip-lifted-to-plane}.

\begin{proposition}\label{prop:good-bip-lifted-distinct-planes}
	Let $u\lessdot w \in S_n$ be permutations such that $uc$ and $wc$ are length-additive and $\Pi_u ^{uc} \cap \Pi_w^{wc}$ is a facet of both polytopes. Then $\langle \lin^w, u^{\hgt} \rangle > \langle \lin^w, w^{\hgt} \rangle$ 
and the local folding condition holds for $\Pi_u ^{uc} \cap \Pi_w^{wc}$.
\end{proposition}

\begin{proof}
	Facets of $\Pi_u ^{uc}$ and $\Pi_w^{wc}$ correspond to subintervals of rank one smaller. So we may assume $u \lessdot w= (a~b) u \le uc \lessdot wc= (a~b) uc$ where $a<b$. We will check that
	\[\langle \lin^w, u^{\hgt}- w^{\hgt}\rangle = (\lin^w_a - \lin^w_b) (\pt_{u^{-1}(a)}- \pt_{u^{-1}(b)}) + \hgt(u)- \hgt(w) >0, \]
	or equivalently that
	\begin{equation}\label{eq:folding-ineq}
		\steep(u, (a~b))=\frac{\hgt(w)- \hgt(u)}{\pt_{u^{-1}(a)}- \pt_{u^{-1}(b)}} <\lin^w_a - \lin^w_b.
	\end{equation}
	
	Since $w \in [u, uc]$, the proof of \cref{prop:good-bip-lifted-to-plane} implies the left hand side of \eqref{eq:folding-ineq} is equal to $\wt(\sigma)$, where $\sigma$ is the ribbon strip of content $[a,b-1]$ contained in $\skew_{u[i]/uc[i]}^{r_i}$ for any $i$ such that $u[i] \neq w[i]$.
	
	On the other hand, $uc \in [w,wc]$ and so by the first item of \cref{prop:ratios-for-coplanarity}, $\steep(uc, (a~b)) = \lin^w_a - \lin^w_b$. 
	Again, the proof of \cref{prop:good-bip-lifted-to-plane} implies that the right hand side is equal to $\wt(\mu)$, where $\mu$ is the ribbon strip of content $[a,b-1]$ contained in $\skew_{w[i]/wc[i]}^{r_i}$ for any $i$ such that $uc[i] \neq wc[i]$.
	
	So it suffices to show that $\wt(\sigma) < \wt(\mu)$. Suppose first that there exists some $i$ such that $u[i] \neq w[i]$ and $uc[i] \neq wc[i]$. Note $w[i]= u[i] \setminus \{a\} \cup \{b\}$, so $\skew_{u[i]/w[i]}^{r_i}$ is a ribbon of content $[a,b-1]$ whose lower boundary is $\path_{u[i]}^{r_i}$. Since ribbons are determined by their lower boundaries and their content, this implies that $\skew_{u[i]/w[i]}^{r_i}$ is equal to $\sigma$. Since $\mu$ consists of the boxes of content $[a, b-1]$ immediately above $\path_{w[i]}^{r_i}$, $\mu$ is the translation of $\sigma$ by $(-1,1)$. This implies that $\wt(\mu)= \wt(\sigma)+ (b-a) >\wt(\sigma)$, since weight increases by one if you move in direction $(-1,1)$.
	
	Now suppose there is no $i \in [n-1]$ such that $u[i] \neq w[i]$ and $uc[i] \neq wc[i]$. Let $x:=u^{-1}_a$ and $y:=u^{-1}_b$, so that $u[i] \neq w[i]$ exactly when $i \in [x, y-1]$ and $uc[i] \neq wc[i]$ exactly when $i \in [c^{-1}_x, c^{-1}_y-1]$. Note that both of these intervals are nonempty, as $u \le w$ implies $x<y$ and $uc \le wc$ implies $c^{-1}_x<c^{-1}_y$. We are assuming that $[x,y-1] \cap [c^{-1}_x, c^{-1}_y -1]$ is empty. This can happen only if $c^{-1}_x<x$, $c^{-1}_y<y$ and $x=c^{-1}_y$, or if $c^{-1}_x>x$, $c^{-1}_y>y$ and $y=c^{-1}_x$. The two cases are similar, so we assume we are in the second case and $[c^{-1}_x, c^{-1}_y -1]= [y, c^{-1}_y-1]$. We will compare the path $\path_{u[y-1]}^{r_{y-1}}$, which is the bottom boundary of $\sigma$, and the path $\path_{w[y]}^{r_{y}}$, the bottom boundary of $\mu$. Since $c_y=x <y$, $y$ is not an excedance of $c$ and $r_{y}= r_{y-1}+(-1,1)$. We also have that $w[y]=u[y]= u[y-1] \cup \{b\}$. So steps $1, \dots, b-1$ of $\path_{w[y]}^{r_y}$ are obtained from those of $\path_{u[y-1]}^{r_{y-1}}$ by translation by $(-1,1)$. Step $b$ of $\path_{w[y]}^{r_y}$ is a step down, while step $b$ of $\path_{u[y-1]}^{r_{y-1}}$ is a step left. So the paths $\path_{u[y-1]}^{r_{y-1}}$ and $\path_{w[y]}^{r_y}$ meet after step $b$ (and then agree). This means that the boxes above $\path_{u[y-1]}^{r_{y-1}}$ and below $\path_{w[y]}^{r_y}$ form a ribbon strip of content $[1, b-1]$. In particular, this ribbon strip contains $\sigma$ and so $\sigma$ is immediately below $\path_{w[y]}^{r_y}$. This again implies that $\mu$ is obtained from $\sigma$ by translation by $(-1,1)$, and so $\wt(\mu)= \wt(\sigma) + (b-a)>\wt(\sigma)$.
	
	This completes the proof that $\langle \lin^w, u^{\hgt} \rangle > \langle \lin^w, w^{\hgt} \rangle$. This inequality implies that $u^{\hgt}$ lies strictly above the affine span of $(\Pi_{w}^{wc})^{\hgt}$, which is cut out by the equation $\langle \lin^w, x-w^h \rangle=0$ and the equation $\langle (1, \dots, 1, 0), x \rangle = \sum_i \pt_i$. This is the local folding condition, so we are done.
\end{proof}

\begin{proof}[Proof of \cref{thm:regular-type-A}]
	\cref{prop:good-bip-lifted-to-plane} shows the coplanarity condition of \cref{thm:regular-subdiv-conditions} holds for 
	\[\{\Pi_u^{uc}: uc \text{ length-additive}\}\]
	when lifted with the height vector $\hc$. \cref{prop:good-bip-lifted-distinct-planes} shows the local folding condition holds.
\end{proof}

\subsection{Relation to matroidal subdivisions, Dressians, tropical Grassmannians}\label{ssec:rel-to-matroid-polytopes} 

We give additional context for the reader familiar with matroid theory. See \cite{BEW24} for additional details. 

For a Bruhat interval $[u,v]$, let $M_{u,v}^k:=\{z[k]: z \in [u,v]\}$. By \cite{KW15}, the set $M_{u,v}^k$ is a \textbf{positroid}, meaning that it is a matroid which can be realized by a $k \times n$ matrix with nonnegative maximal minors. The matroid polytope of $M_{u,v}^k$ is 
\[P_{u,v}^k=\conv\{e_{z[k]} : z \in [u,v]\} \subset \RR^n\]
and is also called a \textbf{positroid polytope}. (Recall that $e_I$ denotes the indicator vector of the subset $I \subset [n]$ in $\RR^n$.) The tuple of matroids $M_{u,v}:=(M_{u,v}^1, \dots, M_{u,v}^{n-1})$ form a flag matroid, whose matroid polytope is the Minkowski sum
\[P_{u,v}=P_{u,v}^1 + \dots + P_{u,v}^{n-1}.\]
The flag matroids $M_{u,v}$ are exactly the \textbf{flag positroids}, meaning they are the only flag matroids which can be realized by an $n \times n$ matrix with nonnegative \emph{flag minors}.

Fix $\pt=(n, n-1, \dots, 2,1)$. One can check that the Bruhat interval polytope $\Pi_u^v(\pt)$ is equal to $P_{u,v}$, the flag positroid polytope. The permutahedron $\Pi_e^{w_0}(\pt)$ is the flag matroid polytope of the \defn{uniform flag matroid}, meaning the constitutent matroids $M_{e, w_0}^k$ are all uniform (and the corresponding matroid polytopes are the hypersimplices). Thus, a regular Bruhat interval subdivision of $\Pi_e^{w_0}(\pt)$ is a regular flag-positroidal subdivision of the uniform flag matroid polytope. This is the ``flag" generalization of regular positroidal subdivisions of hypersimplices. 

Regular positroidal subdivisions of the hypersimplex $\Delta_{i,n}$ were studied in \cite{SW21,LPW,AHLS}. The height vectors $h$ giving rise to such subdivisions form a polyhedral fan, the \emph{positive Dressian} \cite{LPW}, which coincides with the \emph{positive tropical Grassmannian} \cite{SW21,AHLS}. 

An analogous story was developed by \cite{JLLO23,B23} for regular flag-positroidal subdivisions (i.e. Bruhat interval subdivisions) of $\Pi_e^{w_0}(\pt)$. The height vectors $h$ giving rise to such subdivisions form a polyhedral fan called the \emph{positive flag Dressian}, introduced by Joswig--Loho--Luber--Olarte in \cite{JLLO23}. Combining \cite{JLLO23} with work of Boretsky \cite{B23}, one obtains that the positive flag Dressian is equal to the \emph{positive tropical flag variety}. That is, the points of the positive flag Dressian may be obtained by tropicalizing points of the positive flag variety. This was subsequently generalized by Boretsky--Eur--Williams to the ``partial flag" setting \cite{BEW24}.

There is a bijective correspondence between cones of the positive flag Dressian\footnote{Technically, one can endow the positive flag Dressian with a number of different fan structures. Here, we choose the secondary fan structure.} and regular Bruhat interval subdivisions of $\Pi_e^{w_0}(\pt)$. Containment of cones corresponds to coarsening of subdivisions.  We note that, while \cite{B23} gives a parametrization of the positive flag Dressian, very little is known about its secondary fan structure---such as the number of rays or maximal cones---or about which Bruhat interval polytopes form regular Bruhat interval subdivisions. \cref{thm:regular-type-A} gives partial insight into these questions. Each subdivision in \cref{thm:regular-type-A} corresponds to a maximal cone of the positive flag Dressian, and we have a complete description of which Bruhat interval polytopes appear in the subdivision. From \cref{thm:regular-type-A}, we obtain the following (likely quite poor) bound for the number of maximal cones.

\begin{corollary}
	The positive flag Dressian $FlDr^+(n)$, endowed with the secondary fan structure, has at least $2^{n-2}$ maximal cones.
\end{corollary}

\subsection{Numerology of cells}

In this subsection, we discuss the number of cells in the subdivision of \cref{thm:regular-type-A}, which varies depending on the choice of $c$.

As detailed in \cref{ssec:rel-to-matroid-polytopes} above, regular Bruhat interval subdivisions of $\Pi$ are the ``flag" analogue of regular positroidal subdivisions of the hypersimplex $\Delta_{k,n}$. In the latter context, all finest subdivisions have the same $f$-vector, and in particular have $\binom{n-2}{k}$ maximal cells. In \cite{BEW24}, it was shown for $n=4$ that, in contrast, finest regular Bruhat interval subdivisions of $\Pi$ need not have the same number of maximal cells. Here we show this continues for arbitrary $n$.

We use $\le_L$ to denote the left weak order on $S_n$. Recall that $u \le_L v$ if and only if $v$ has a reduced expression which has a reduced expression for $u$ as a suffix. Equivalently, $u \le_L v$ if and only if $v$ can be written as $v=xu$ where $xu$ is length-additive. Left weak order can also be rephrased in terms of right inversions.

\begin{definition}
	For $u \in S_n$, let $T_R(u):= \{(a,b): a<b, u_a > u_b\}$ be the set of \emph{right inversions} of $u$. We define $\succcurlyeq_u$ to be the poset on $[n]$ defined by $a \succcurlyeq_u b$ if $(a,b) \in T_R(u)$ or if $a=b$.
\end{definition}

It is straightforward to verify that $\succcurlyeq_u$ is indeed a partial ordering on $[n]$. It follows from \cite[Proposition 3.1.3]{BB05} that for $u \in S_n$
\[[u, \wo]_L= \{v \in S_n: v(a)>v(b) \text{ for all }a \neq b \text{ with }a \succcurlyeq_u b \}\]
where the set on the right is exactly the set of linear extensions of $\succcurlyeq_u$.

\begin{lemma} \label{lem:cells-bijection-weak-order-interval} Let $c \in S_n$ be a Coxeter element.
	Then 
	\begin{align*}
		\{v: c \le_L v\} &\to \{\Pi_w^{wc}: wc \text{ length-additive}\}\\
		v &\mapsto \Pi_{vc^{-1}}^v
	\end{align*}
	is a bijection. That is, the maximal cells of the $\hc$-subdivision of $\Pi$ are in bijection with the left weak order interval $[c, \wo]_{L}$ and the linear extensions of $\succcurlyeq_c.$
\end{lemma}

\begin{corollary}
	Let $c \in S_n$ be a Coxeter element. The number of maximal cells in the $\hc$-subdivision of $\Pi$ is the number of linear extensions of $\succcurlyeq_c$.
\end{corollary}

\begin{figure}[h]
	\centering
	\includegraphics[width=0.9\textwidth]{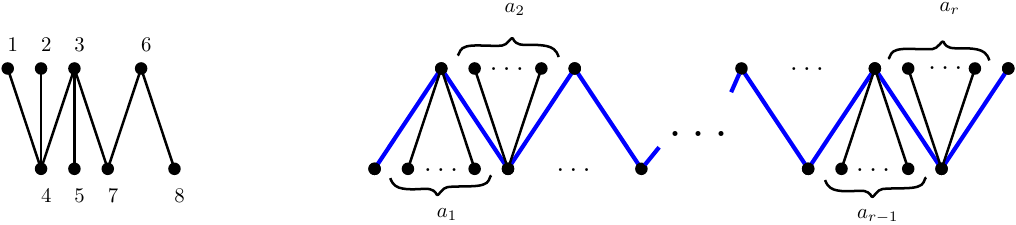}
	\caption{On the left, the Hasse diagram of $\succcurlyeq_c$ when $c=s_5s_7s_6s_1s_2s_4s_3=23614857$. On the right, the general shape of the Hasse diagram of $\succcurlyeq_c$, which is a zig-zag (in blue) with claws attached to the peaks and valleys. The numbers $a_1, \dots, a_r$ are arbitrary nonnegative integers.}
	\label{fig:Hasse-ex}
\end{figure}

See \cref{fig:Hasse-ex} for an example of the poset $([n],\succcurlyeq_c)$, as well as their general ``shape". Note that the posets $([n],\succcurlyeq_c)$ are height 2 tree posets on $[n]$. As a result, the number of linear extensions may be computed using a polynomial-time algorithm \cite{Atk90}.

\begin{example}The maximum number of cells in an $\hc$-subdivision of $\Pi$ is achieved when $c=s_1s_2 \dots s_{n-1}= 234\cdots n 1$ or its inverse. In this case, $([n],\succcurlyeq_c)$ is a \emph{claw}, and has $(n-1)!$ linear extensions. This subdivision was first studied by \cite{HHMP19}, who showed that each maximal cell is combinatorially a cube.
	
	The number of cells likely achieves its minimum when $c=\prod_{i \text{ even}} s_i \prod_{i \text{ odd}} s_i$ or its inverse. In this case, $([n],\succcurlyeq_c)$ is a \emph{zig-zag poset}. Linear extensions are in bijection with ``up/down" permutations and are counted by Euler numbers (see \cref{tab:Euler-vs-factorial}).
\end{example}

\begin{table}[h]
	\centering
	\begin{tabular}{c|c|c|c|c|c|c|c|}
		$n$ & 3 & 4 & 5 & 6 & 7 & 8 & 9\\
		\hline
		$(n-1)!$ & 2 & 6 & 24 & 120 & 720 & 5040 & 40320\\
		\hline
		Euler number & 2 & 5 & 16 & 61 & 272 & 1385 & 7936 \\
		\hline
	\end{tabular}
	\caption{Ranging over all Coxeter elements $c$ in $S_n$, the largest possible number of maximal cells in an $\hc$-subdivision of $\bip{e, \wo}$ is $(n-1)!$. The smallest possible number of maximal cells is the $n$th Euler number. The two numbers are contrasted above.}
	\label{tab:Euler-vs-factorial}
\end{table}

\section{Regularity in types $B/C$}\label{sec:regularity-BC}

Here we use the technique of \emph{folding} \cite{Stembridge} to deduce the regularity of the subdivision in \cref{thm:subdivision-Coxeter} for types $B$ and $C$ from regularity in type $A$. Note that Bruhat interval polytopes for types $B$ and $C$ are the same (though the integral Bruhat interval polytopes of each type are different). So throughout the section, we will deal with type $C$.

We fix an isomorphism of the type $C_n$ realified weight lattice
$\mathfrak{t}_{\RR}^*$ with $\RR^n$ so that the fundamental weights
are $ \varpi_i=e_1 + \cdots +e_i$ for $i=1, \dots, n$, where $e_i$ is
the standard basis vector. The type $C_n$ Weyl group $W$ is generated
by $s^C_1,\ldots, s^C_{n-1}, s^C_n$, which act on $\RR^n$ by
$s_i^C \cdot \pt = (\pt_1, \dots, \pt_{i+1}, \pt_i, \dots \pt_n)$ for
$i=1, \dots, n-1$ and $s_n^C \cdot \pt=(\pt_1, \dots, -\pt_n)$. Let
$\iota: W \to S_{2n}$ be the map defined on generators by
$\iota(s_i^{C}) = s_is_{2n-i}$ for $1 \leq i \leq n-1$ and
$\iota(s^C_n) = s_{n}$. For brevity, we write $\bar{w}$ for
$\iota(w)$. We note that $\iota$ takes reduced words to reduced words,
and so is Bruhat-order preserving.

For the remainder of the section, fix $c \in W$ a Coxeter element and $\pt =(\pt_1, \dots, \pt_n) \in \BB$. We define $\bar{\pt}:=(\pt_1 \dots, \pt_n, -\pt_n, \dots, -\pt_1) \in \RR^{2n}$. 

We first use $\iota$ to define a height function $\hc$ on the vertices of $\Pi(\pt)= \Pi$.

 \begin{definition}\label{def:type-C-height}
 	Let $c \in W$ be a Coxeter element, and note $\bar{c} \in S_{2n}$ is also a Coxeter element. Let $\hd: (S_{2n} \cdot \bar{\pt}) \to \RR$ denote the height function of \cref{def:height-from-c}. We define $\hc:(W \cdot \pt) \to \RR $ by 
 	\[\hc(w \cdot \pt) := \hd(\bar{w} \cdot \bar{\pt}).\]
 \end{definition}
 
 This height function does indeed induce the desired subdivision of $\Pi$.
 
 \begin{theorem}\label{thm:regular-type-C}
   Let $\pt \in \BB$, let $c \in W$ be a Coxeter element, and let
   $\hc$ be the height function from \cref{def:type-C-height}. The collection
   \[ \{\Pi_u^{uc}(\pt): uc~\text{length-additive}\} \]
   is a finest regular Bruhat interval subdivision of the type $C$
   permutahedron $\Pi(\pt)$, induced by $\hc$.
 \end{theorem}

 We have the following useful lemma. For a linear functional $\lambda: \RR^{2n+1} \to \RR$, let $\lambda': \RR^{n+1} \to \RR$ be the linear functional with $\lambda'_{n+1}=\lambda_{2n+1}$ and $\lambda'_i= \lambda_i-\lambda_{2n-i+1}$. To ease notation, set $w^h:=\hc(w \cdot \pt)$ and $\bar{w}^{\bar{\hgt}}:=\hd(\bar{w} \cdot \bar{\pt})$.
 
 \begin{lemma} \label{lem:folded-linear-functional}
 	Let $\lambda: \RR^{2n+1} \to \RR$ be a linear functional and $w \in W$. Then $\langle \lambda, \bar{w}^{\bar{\hgt}} \rangle = \langle \lambda', w^h\rangle$.
 \end{lemma}
 \begin{proof} Say $w \cdot \pt=(x_1, \dots, x_n)$. Then $\bar{w} \cdot \bar{\pt} = (x_1, \dots, x_n, -x_n, \dots, -x_1)$. We have 
 	
 	\[\langle \lambda, \bar{w}^{\bar{\hgt}} \rangle = x_1 (\lambda_1 - \lambda_{2n}) + \cdots + x_n (\lambda_n - \lambda_{n+1}) + \hd(\bar{w} \cdot \bar{\pt}) \lambda_{2n+1}= \langle \lambda', w^h \rangle\]
 	since by definition, $\hc(w \cdot \pt) = \hd(\bar{w} \cdot \bar{\pt}).$
 \end{proof}
 
 \begin{proof}[Proof of \cref{thm:regular-type-C}]
 	We again verify the conditions of \cref{thm:regular-subdiv-conditions}.
 	
 	\noindent {\em The coplanarity condition.} 	Fix $u$ where $uc$ is length-additive. The product $\bar{u} \bar{c}$ is also length-additive, since $\iota$ takes reduced words to reduced words. Thus, by \cref{prop:good-bip-lifted-to-plane}, there is a linear functional $\lambda^u$ which is constant on $\{w^{\bar{h}}: w \in [\bar{u}, \bar{uc}]\}$. If $v \in [u,uc]$, then $\bar{v} \in [\bar{u}, \bar{uc}]$. So by \cref{lem:folded-linear-functional}, the linear functional $(\lambda^u)'$ is constant on $\{v^h: v \in [u,uc]\}$.
 	
 	\noindent {\em The local folding condition.} Suppose $\Pi_u^{uc} \cap \Pi_w^{wc}$ is a facet of both polytopes. Without loss of generality, $w \lessdot u \leq wc \lessdot uc$ and $\Pi_u^{uc} \cap \Pi_w^{wc} = \Pi_u^{wc}$. So $w \cdot \pt$ is not a vertex of both polytopes. Now, $\Pi_{\bar{u}}^{\bar{uc}}(\bar{\pt})$ and $\Pi_{\bar{w}}^{\bar{wc}}(\bar{\pt})$ are maximal cells of a regular subdivision of the type $A$ permutahedron. The point $\bar{w} \cdot \bar{\pt}$ is not a vertex of, and consequently is not contained in, $\Pi_{\bar{u}}^{\bar{uc}}(\bar{\pt})$. So the linear functional $\lambda^u$ which is minimized on the lifted polytope  $(\Pi_{\bar{u}}^{\overline{uc}}(\bar{\pt}))^{\bar{h}}$ does not take its minimum value on $\bar{w}^{\bar{h}}$. So we have
 	\[\langle \lambda^u, \bar{u}^{\bar{h}} \rangle < \langle  \lambda^u, \bar{w}^{\bar{h}} \rangle\]
 	which by \cref{lem:folded-linear-functional} implies 
 		\[\langle (\lambda^u)', {u}^{{h}} \rangle < \langle  (\lambda^u)', {w}^{{h}} \rangle.\]
 		That is, a vertex of $(\Pi_w^{wc})^h$ lies above the plane containing $(\Pi_u^{uc})^h$, as desired.
 \end{proof}

\section{On Kac-Moody analogues}

The moment polyhedron of a Kac-Moody flag ``variety'' (really, an ind-scheme)
has received a fair amount of study e.g.
in \cite{Yun,HHH,BaumannKamnitzerTingley}.

\begin{conjecture}
  Let $G$ be a Kac-Moody group of finite rank, and $c$ a Coxeter element.
  Then the moment polyhedron $\Phi_T(G/B)$ has a locally finite subdivision
  into Bruhat interval polytopes 
  \[\{\Phi_T(X_w^{wc}): wc\text{ length-additive}\}.\]
\end{conjecture}

We draw (a finite part of) the only new $2$-dimensional example,
where $G = \widehat{SL_2}$, whose Weyl group is
$\langle r,s \colon r^2 = s^2 = 1 \rangle$. 
The vertices are at $\{(n,n^2)\colon n\in \ZZ\}$, where the coordinates
correspond to the $T^1 \leq SL_2$ and loop rotation circle actions.

\centerline{
\begin{tikzpicture}[scale=1.2]
  
  \coordinate (P0) at (-4, 4);
  \coordinate (P1) at (-3, 2.25);
  \coordinate (P2) at (-2, 1.00);
  \coordinate (P3) at (-1, 0.25);
  \coordinate (P4) at (0,  0.00);
  \coordinate (P5) at (1,  0.25);
  \coordinate (P6) at (2,  1.00);
  \coordinate (P7) at (3,  2.25);
  \coordinate (P8) at (4, 4);
  
  \draw[thick, blue, fill=blue!5] (P0) -- (P1) -- (P2) -- (P3) -- (P4) -- (P5) -- (P6) -- (P7) -- (P8);
  
  \filldraw (P0) circle (1.5pt) node[above left]  {$rsrs$};
  \filldraw (P1) circle (1.5pt) node[above left]  {$rsr$};
  \filldraw (P2) circle (1.5pt) node[left]        {$rs$};
  \filldraw (P3) circle (1.5pt) node[below left]  {$r$};
  \filldraw (P4) circle (1.5pt) node[below]       {$e$};
  \filldraw (P5) circle (1.5pt) node[below right] {$s$};
  \filldraw (P6) circle (1.5pt) node[right]       {$sr$};
  \filldraw (P7) circle (1.5pt) node[above right] {$srs$};
  \filldraw (P8) circle (1.5pt) node[above right] {$srsr$};
  
  \draw[thick, blue] (P5) -- (P2) -- (P7) -- (P0);

  \draw[->, thick, blue] (P0) -- (-4.5,5.125);
  \draw[->, thick, blue] (P8) -- (4.5,5.125);

  \draw[-] (-4,0) -- (4,0) ; 
  \draw[-] (0,0) -- (0,5) ; 
\end{tikzpicture}
}

There are several difficulties in trying to extend our results to the
infinite-dimensional case. The first is that there is no ``general enough''
point $p$ in the thin flag variety such that $\Phi_T(\overline{T\cdot p})$ is
the entire moment polyhedron $\Phi_T(G/B)$. This is easy to see; as the
thin flag variety is ind-projective, any $p$ lies inside a projective variety,
so $\overline{T\cdot p}$ is projective and has finitely many $T$-fixed points.
Perhaps this could be addressed by working with the thick flag variety.

More plausible is the existence of a degeneration
$\Omega_c \to \Union_w \left( X_w \times X^{wc} \right)$, even though
neither side is finite-dimensional or finite-codimensional.  What they
both do enjoy is that the projection to the first factor has
finite-codimensional image and finite-dimensional fibers, like the
spaces in \cite{BFN1}.
 
\bibliographystyle{alpha}
\bibliography{refs.bib}

\appendix
\end{document}